\documentclass{article}
\usepackage{amsmath}
\usepackage{amssymb}
\usepackage{amsthm}
\usepackage{graphicx}
\usepackage[affil-it]{authblk}
\newtheorem{theorem}{Theorem}[section]
\newtheorem{lemma}[theorem]{Lemma}
\newtheorem{proposition}[theorem]{Proposition}
\newtheorem{definition}[theorem]{Definition}
\newtheorem{corollary}[theorem]{Corollary}

\title{Reducibility of varieties of commuting $m$-tuples of elements of simple Lie algebras}
\author{Nikola Kova\v{c}evi\'{c} \footnote{This paper is written as part of doctoral studies at Faculty of Mathematics and Physics, University of Ljubljana}\\
e-mail: nikola.kovacevic@imfm.si}
\affil{Institute of Mathematics, Physics and Mechanics,\\
  Jadranska ulica 19\\
  1000, Ljubljana\\
   Slovenia}
\affil{University of Ljubljana,\\
Faculty of Mathematics and Physics,\\
Jadranska ulica 19\\
1000, Ljubljana\\
Slovenia}
\begin{document}
\maketitle
\textbf{Abstract}\\
\\

In this paper we prove that the variety $C_m(L)$ of commuting $m$-tuples of elements of a simple Lie algebra $L$ is usually reducible. More precisely, one of the main results of this paper shows that if $m = 3$ and $L$ is of types $B,C,E_7,E_8,F_4,G_2$ or $D_l$ with $l \geq 10$, then $C_m(L)$ is reducible. Moreover, if $m \geq 4$ and $L$ is a simple Lie algebra different from $\mathfrak{sl}_2,\mathfrak{sl}_3$, then $C_m(L)$ is reducible. In order to obtain these results, we need a more general result, which is another main result of this paper. It states that, under some mild assumptions, $C_m(L)$ is reducible whenever $C_m(L')$ is reducible for some simple Lie subalgebra $L'$ of $L$ whose Dynkin diagram is a subdiagram of the Dynkin diagram of $L$.\\

\section{Introduction}

Let $m \in \mathbb{N}$ and let $L$ be a Lie algebra. We define:\\
\begin{equation*} C_m(L) = \{(x_1,...,x_m) \in L^m: [x_i,x_j] = 0\ (\forall 1 \leq i,j \leq m)\}.\\ \end{equation*}
This is an affine variety, called \textbf{$m$-th commuting variety of $L$}.\\

A research of commuting varieties is important for module varieties (see for example \cite{Z2},\cite{Ngo10},\cite{Sroer}), modular representation theory \cite{Ngo},\cite{Ngo2},\cite{Suslin1},\cite{Suslin2}, fibrations on surfaces \cite{Hitchin}, invariant theory (see, for example \cite{inv}, \cite{Vaccarino}), etc. Moreover, in the case of $L = \mathfrak{gl}_n$, the commuting variety is related to Hilbert and Quot schemes \cite{Baranovski},\cite{ADHM},\cite{Commuting},\cite{Jel},\cite{Nakajima}, and tensors \cite{JelLanPal},\cite{Landsberg}. Commuting varieties of parabolic subalgebras of $\mathfrak{gl}_n$ are connected to nested Hilbert schemes, see \cite{boosbulois} or \cite{nest}.\\

Throughout the paper we will work over algebraically closed field $F$ of characteristic zero. We can ask basic geometric questions about $C_m(L)$, for example for which $m$ and which Lie algebra $L$, $C_{m}(L)$ is reducible, and what are the irreducible components in the reducible cases. This question has mostly been studied in the case $m = 2$. By the result of Richardson \cite{Richardson}, $C_{2}(L)$ is irreducible for any reductive Lie algebra $L$. This is not necessarily true if $L$ is not reductive. See \cite{boosbulois}, \cite{nest}, \cite{goddardgodwin} for the cases where $L$ is parabolic. On the other hand, much less is known for $m \geq 3$. \\

In type $A$, the results about irreducibility of $C_m(L)$ are usually stated for $L = \mathfrak{gl}_l$ but they clearly hold also for $\mathfrak{sl}_l$. It is known that $m$-tuples of commuting $l \times l$ matrices form reducible variety for $m,l \geq 4$ \cite{Guralnick}. On the other hand, $C_m(\mathfrak{sl}_2)$, $C_m(\mathfrak{sl}_3)$ are irreducible for all $m$, see \cite{Guralnick}. See also \cite{Ngoo} for other geometrical properties of these varieties. The question of irreducibility of commuting varieties is open in type $A$ only in the case $m=3$. For triples of $l\times l$ matrices, we know that $C_{3}(\mathfrak{sl}_{l})$ is reducible for $l \geq 29$ \cite[Theorem 7.9.3]{Ohara} \cite{HolOml}, while it is irreducible for $l \leq 10$ \cite{Yongo,Sivic1, Sivic3}. The question is still open for $11 \leq l \leq 28$. Irreducible components of $C_m(\mathfrak{sl}_l)$ are classified in \cite{Jel} for $l \leq 7$ and all $m$.\\

Generalisation of the above results to other simple Lie algebras is a natural problem. Besides that, we believe that a research of commuting varieties in the general context of simple Lie algebras may give some answers to the open questions in type $A$ and corresponding Hilbert and Quot schemes. This paper is the first step in this direction, where we study irreducibility of $C_m(L)$. To the best of our knowledge, irreducibility of $C_m(L)$ for $L$ being a simple Lie algebra not of type $A$ has not been studied yet. The only results in this direction could be obtained by modifying results from \cite{Ngo2} to conclude that $C_{m}(L)$ is often reducible for classical $L$, but the so obtained results would not be sharp. In particular, the results in \cite{Ngo2} give no information about the case $m = 3$. The results in this paper will be sharper than those that could be obtained from \cite{Ngo2} because we use different techniques. One of our main tools in proving reducibility is the investigation of the tangent space to $C_m(L)$. This method was first applied to commuting varieties in [19], but it has never been used for $L$ different from $\mathfrak{gl}_l$.\\

In the paper we first consider the question whether reducibility of $C_m(L')$ for some simple Lie subalgebra $L'$ of $L$ whose Dynkin diagram is a subdiagram of the Dynkin diagram of $L$ implies reducibility of $C_m(L)$. It is natural to expect that this holds, but it is nontrivial to prove it. For example, if $L$ is of type $A$, this result is proved in \cite{HolOml} using functional calculus on matrices and in \cite{Jel} it is proved using obstruction theory. In the paper we prove this result under mild assumptions for all types using only elementary algebraic geometry. More precisely, we show that $C_m(L)$ is reducible if either $C_m(L')$ has a nonregular component (see Section 2.4 for the definition) of dimension at least $\dim(L')+(m-1)\mathrm{rank}(L')$ or if $C_m(L')$ has a component that is generically reduced as a scheme. However, we present an elementary proof which does not use scheme theory.\\

We call the above result Theorem on Adding Diagonals, as in type $A$ the obtained component of $C_m(L)$ is the closure of the $G$-orbit of the set of $m$-tuples of matrices that are direct sums of matrices from the component of $C_m(L')$ and diagonal matrices, as in \cite[Proposition 2.1]{HolOml}. For example, in Section 4.2 we show that the closure of the orbit of the triple\\
\[ \left( \left[\begin{array}{cccc}
0 & 1 & 0 & 0\\
0 & 0 & 0 & 0\\
0 & 0 & 0 & 0\\
0 & 0 & -1 & 0
\end{array}\right], \left[\begin{array}{cccc}
0 & 0 & 1 & 0\\
0 & 0 & 0 & 0\\
0 & 0 & 0 & 0\\
0 & 0 & 0 & 0
\end{array}\right], \left[\begin{array}{cccc}
0 & 0 & 0 & 0\\
0 & 0 & 0 & 0\\
0 & 0 & 0 & 0\\
0 & 1 & 0 & 0
\end{array}\right] \right)\]\\
is a component of $C_3(\mathfrak{sp}_4)$. Theorem on adding diagonals then shows that $C_3(L)$ is reducible for $L$ of type $F_4$ or $C_k, B_k$ for all $k \geq 2$.\\

The second main result, that we get by applying Theorem on Adding Diagonals, is the following.\\

\begin{theorem} Let $L$ be a simple Lie algebra. Then $C_m(L)$ is reducible in the following cases:\\
(1) $m \geq 4$ and $L$ not isomorphic to $\mathfrak{sl}_2$ or $\mathfrak{sl}_3$.\\
(2) $m = 3$, and $L$ of type $B_k,C_k,E_7,E_8,F_4,G_2$ or $D_l$ with $l \geq 10$.\end{theorem}
 
The investigation of the remaining cases seems to be more technical, as in \cite{Yongo,Sivic1,Sivic3} for type $A$. We plan to investigate them in the next papers.\\

The paper is organized as follows: In section $2$ we present some auxiliary results which we need in the proof of theorem on adding diagonals. This result is proved in section $3$. In section $4$ we see how the theorem on adding diagonals helps us to deduce that $C_{m}(L)$ is reducible for all $m$ and $L$ as in Theorem 1.1.\\

\section{Preliminaries}

\subsection{Notation}

Let $F$ be an algebraically closed field of characteristic zero. Let $G$ be a connected simple algebraic group and let $L$ be its Lie algebra, denoted by $\mathfrak{L}\mathfrak{i}\mathfrak{e}(G)$. Let $H$ be a Cartan subalgebra of $L$ and let $\Phi$ be the root system associated to the pair $(L,H)$. Let $\Delta$ be a base of the root system $\Phi$. In our work we often use a Chevalley basis of $L$: $\{x_{\alpha},h_{\beta}: \alpha \in \Phi,\beta \in \Delta\}$, see \cite[Theorem 25.2]{Hump}. We denote with $R_{\alpha,\beta}$ the constant satisfying  $[x_{\alpha},x_{\beta}] = R_{\alpha,\beta}x_{\alpha+\beta}$, if $\alpha+\beta \in \Phi$. If $\alpha+\beta \notin \Phi$ and $\beta \neq -\alpha$, we denote $R_{\alpha,\beta} = 0$.\\
For an element $x \in L$, we denote the semisimple part of $x$ with $x_s$ and the nilpotent part by $x_n$.

Choose an irreducible subset $\Delta' \subset \Delta$, which is a base of the irreducible root system $\Phi' = \Phi \cap \mathrm{span}(\Delta')$, define\\
\begin{equation*} H' = \mathrm{span}(h_{\alpha};\alpha \in \Delta'),\\ \end{equation*}
and let\\
\begin{equation*} L' = H' + \sum_{\alpha \in \Phi'} L_{\alpha},\\ \end{equation*}
where $L_{\alpha}$ is the root space of $\alpha$. Then $L'$ is semisimple, see Corollary \ref{aa}, and since $\Phi'$ is irreducible, $L'$ is simple. 

If $G$ is an algebraic group and $K$ its algebraic subgroup, we denote the connected component of $K$ that contains identity with $K^{\circ}$. Algebraic group $G$ acts on its Lie algebra $L = \mathfrak{L}\mathfrak{i}\mathfrak{e}(G)$ by the adjoint action, and we denote that action with $\cdot$. This action can be extended to $C_m(L)$:\\ 
\begin{equation*} \cdot: G \times C_m(L) \to C_m(L),\\ \end{equation*}
\begin{equation*} (g,x_1,...,x_m) \mapsto (g \cdot x_1,...,g \cdot x_m).\\ \end{equation*}

Let $\mathcal{L}$ be a reductive Lie algebra, and let $\mathcal{G}$ be a connected algebraic group with $\mathfrak{Lie}(\mathcal{G}) = \mathcal{L}$. Let $x \in \mathcal{L}$, let $W \subset \mathcal{L}$ be an arbitrary subset, and let $L_1$ be a Lie subalgebra of $\mathcal{L}$. We denote the center of $L_1$ with $Z(L_1)$. We define:\\
\begin{equation*} N_{\mathcal{G}}(W) = \{g \in \mathcal{G}: g \cdot w \in W,\ (\forall w \in W)\},\\ \end{equation*}
\begin{equation*} \mathrm{Stab}_{\mathcal{G}}(x) = \{g \in \mathcal{G}: g \cdot x = x\},\\ \end{equation*}
\begin{equation*} \mathrm{Stab}_{\mathcal{G}}(W) = \cap_{x \in W} \mathrm{Stab}_{\mathcal{G}}(x),\\ \end{equation*}
\begin{equation*} N_{\mathcal{L}}(W) = \{x \in \mathcal{L}: [x,w] \in W\ (\forall w \in W)\},\\ \end{equation*}
\begin{equation*} C_{\mathcal{L}}(x) = \{y \in \mathcal{L}: [x,y] = 0\},\\ \end{equation*}
\begin{equation*} C_{\mathcal{L}}(W) = \cap_{x \in W} C_{\mathcal{L}}(x), \\ \end{equation*}
\begin{equation*} C_{L_1}(x) = C_{\mathcal{L}}(x) \cap L_1,\\ \end{equation*}
\begin{equation*} C_{L_1}(W) = C_{\mathcal{L}}(W) \cap L_1,\\ \end{equation*}
\begin{equation*} N_{L_1}(W) = N_{\mathcal{L}}(W) \cap L_1.\\ \end{equation*}

\subsection{Basic results about simple Lie algebras and algebraic groups}

For the theory of algebraic groups and Lie algebras, we refer to the standard textbooks, such as \cite{Hosild},\cite{Hump},\cite{TauvelYu}. \\

It is very useful to collect in a single proposition all results on the correspondence between algebraic groups and Lie algebras from the literature that we need. These results are presented in \cite[24.3.6]{TauvelYu}, \cite[23.6.3]{TauvelYu}, \cite[24.3.5]{TauvelYu} and \cite[Theorem VIII.3.4]{Hosild}.\\

\begin{proposition}\label{b} Let $\mathcal{L}$ be a reductive Lie algebra, let $\mathcal{G}$ be a connected algebraic group with $\mathfrak{Lie}(\mathcal{G}) = \mathcal{L}$ and let $W$ be a Lie subalgebra of $\mathcal{L}$. With notation introduced in Section 2.1 we have:\\
(1) $\mathfrak{Lie}(N_{\mathcal{G}}(W))  = N_{\mathcal{L}}(W)$\\
(2) $\mathfrak{Lie}(\mathrm{Stab}_{\mathcal{G}}(x)) = C_{\mathcal{L}}(x)$, for all $x \in \mathcal{L}$.\\
(3) $\mathfrak{Lie}(\mathrm{Stab}_{\mathcal{G}}(W)) = C_{\mathcal{L}}(W)$\\
(4) If $g \in \mathcal{G}, x \in \mathcal{L}$ are arbitrary, then $(g \cdot x)_s = g \cdot x_s$ and $(g \cdot x)_n = g \cdot x_n$.\\
(5) For algebraic subgroups $H,K$ of $\mathcal{G}$, we have: $\mathfrak{Lie}(H \cap K) = \mathfrak{Lie}(H) \cap \mathfrak{Lie}(K)$.\\
(6) Suppose that for connected algebraic subgroups $H,K$ of $\mathcal{G}$ we have $\mathfrak{Lie}(H) \subset \mathfrak{Lie}(K)$. Then $H \subset K$.\\
(7) Let $U$ and $V$ be irreducible algebraic subgroups of $\mathcal{G}$, and let $\overline{<U,V>}$ be the Zariski closure in $\mathcal{G}$ of the subgroup generated by $U$ and $V$. Then $\mathfrak{Lie}(\overline{<U,V>})$ coincides with the Lie subalgebra of $\mathcal{L}$ generated by $\mathfrak{Lie}(U)$ and $\mathfrak{Lie}(V)$.\end{proposition}

In the rest of this subsection $L$ and $L'$ are as in Section 2.1.\\
 
\begin{corollary}\label{u} Let $x \in L$. Then\\
\begin{equation*} \dim(\mathrm{Stab}_G(x)) = \dim(C_L(x)).\\ \end{equation*}
\end{corollary}

\begin{lemma}\label{dodavanje} If $x = x_{1} + h$ where $h$ is the semisimple part of $x$ and $x_1$ is the nilpotent part of $x$, then\\
\begin{equation*} \dim(C_L(x_{1} + h)) = \dim(C_L(h) \bigcap C_L(x_{1})). \end{equation*}\\ \end{lemma}

\begin{proof} Corollary \ref{u} implies that $\dim(C_L(x_1+h)) = \dim(\mathrm{Stab}_G(x_1+h))$, therefore we consider the stabilizer of $x_1+h$. If we have $g \cdot (x_1+h) = x_1+h$, for some $g \in G$, then the nilpotent part of the left side is $g \cdot x_1$ and of the right side is $x_{1}$, so  $g \in \mathrm{Stab}_G(x_1)$. Similarly, $g \in \mathrm{Stab}_G(h)$, so $g \in \mathrm{Stab}_G(h) \bigcap \mathrm{Stab}_G(x_1)$. We proved that $\mathrm{Stab}_G(h) \bigcap \mathrm{Stab}_G(x_{1}) \supset \mathrm{Stab}_G(x_{1} + h)$, while the other inclusion is obvious. Taking Lie algebras, due to Proposition \ref{b}, we have that $C_L(x_1+h) = C_L(x_1) \bigcap C_L(h)$, hence the dimensions are equal.\\ \end{proof}

\begin{definition}\label{c} We define:\\
\begin{equation*} H_1 = C_H(L').\\ \end{equation*}
\end{definition}

We examine basic properties of $H_1$ that we need in further work. In order to do that, we need the following well-known lemmas from linear algebra.\\

\begin{lemma}\label{g} Let $V$ be a vector space and let $U,V_1,...,V_n$ be vector subspaces of $V$ with $U \subset \cup_{i=1}^n V_i$. Then there exists $k \in \{1,2,...,n\}$ such that $U \subset V_k$.\end{lemma}

\begin{lemma}\label{h}\cite[Lemma 2.4.3]{Pedersen} For a set $\{f, f_1,...,f_n\}$ of linear functionals on a vector space $V$ the following are equivalent:\\
(1) $f = \sum_{k=1}^n \alpha_kf_k$, where $\{\alpha_1,...,\alpha_n\} \subset F$.\\
(2) $\cap_{i=1}^k \ker(f_k) \subset \ker(f)$.\\ \end{lemma}

Existence of the following special element in $H_1$ is crucial for the proof of Theorem on Adding Diagonals.\\

\begin{lemma}\label{f} There is $h \in H_1$ with $C_L(h) = H + L'$.\\ \end{lemma}

\begin{proof} We take a Chevalley basis $\{x_{\alpha},h_{\beta}: \alpha \in \Phi,\beta \in \Delta\}$. Recall that for $h \in H$ and $x = \sum_{\alpha \in \Phi} a_{\alpha}x_{\alpha}+ \sum_{\beta \in \Delta} b_{\beta}h_{\beta}$ we have $[h,x] = \sum_{\alpha \in \Phi} a_{\alpha}\alpha(h)x_{\alpha}$, so $h$ commutes with $x$ if and only if it commutes with each $x_{\alpha}$ with $a_{\alpha} \neq 0$. We are searching for $h \in H$ such that it commutes with every $x_{\alpha}$ for $\alpha \in \Phi'$ and does not commute with any of the $x_{\beta}$ with $\beta \in \Phi \setminus \Phi'$.\\

Using $[h,x_{\alpha}] = \alpha(h)x_{\alpha}$ we see that $[h,x_{\alpha}] = 0$ if and only if $\alpha(h) = 0$, hence we are searching for\\
\begin{equation*} h \in \cap_{\alpha \in \Phi'} \ker(\alpha)\ and\ h \notin \cup_{\beta \in \Phi \setminus \Phi'} \ker(\beta).\\ \end{equation*}

Suppose that such $h$ does not exist, i.e. that\\
\begin{equation*} \cap_{\alpha \in \Phi'}\ker(\alpha) \subset \cup_{\beta \in \Phi \setminus \Phi'} \ker(\beta).\\ \end{equation*}
By Lemma \ref{g} there exists $\beta \in \Phi \setminus \Phi'$ with\\
\begin{equation*} \cap_{\alpha \in \Phi'}\ker(\alpha) \subset \ker(\beta).\\ \end{equation*}
As $\alpha,\beta$ are functionals, we apply Lemma \ref{h} and conclude that there exist $a_{\alpha} \in F$ with\\
\begin{equation*} \beta = \sum_{\alpha \in \Phi'} a_{\alpha} \alpha \in \mathrm{span}(\Phi') \cap \Phi = \Phi',\\ \end{equation*}
which is impossible since $\beta \in \Phi \setminus \Phi'$. This contradiction shows that our $h$ exists.\\ \end{proof}

We denote with $\mathfrak{I}$ the set of all $h \in H_1$ with $C_L(h) = L' + H$. It is nonempty due to Lemma \ref{f}.\\

Lemma \ref{f} together with known results on centralizers of semisimple elements implies that $L'$ is simple:\\

\begin{corollary}\label{aa} $L'$ is a simple Lie algebra.\\ \end{corollary}

\begin{proof} Let $h \in \mathfrak{I}$. Then, by \cite[Lemma 2.1.2]{Col} $C_L(h)$ is reductive and we know that $C_L(h) = H+L'$.
Note that $\alpha(h) = 0$ for $\alpha \in \Phi$ if and only if $x_{\alpha} \in C_L(h) = L'+H$, which holds if and only if $\alpha \in \Phi'$. Moreover, $H_1 = \cap_{\alpha \in \Phi'} \ker(\alpha)$, so the proof of \cite[Lemma 2.1.2]{Col} implies:\\
\begin{equation*} C_L(h) = H_1 \oplus H' \oplus \sum_{\alpha \in \Phi'} L_{\alpha} = H \oplus \sum_{\alpha \in \Phi'} L_{\alpha}.\\ \end{equation*}
In particular, $H_1 \oplus H' = H$.\\
Now, by the proof of \cite[Lemma 2.1.2]{Col}, we get that\\
\begin{equation*} [C_L(h),C_L(h)] =  H' + \sum_{\alpha \in \Phi'}L_{\alpha} = L'\\ \end{equation*}
is semisimple. Moreover, $\Phi'$ is an irreducible root system corresponding to $L'$, therefore $L'$ is simple.\\ \end{proof}

The above proof immediately implies the following:\\
 
\begin{lemma}\label{d} (1) $H_1 \cap L' = \{0\}$.\\
(2) $H = H_1 \oplus H'$\\
(3) $\dim(H_1) = |\Delta|-|\Delta'|$.\\ \end{lemma}

\begin{corollary}\label{i} The set $\mathfrak{I}$ is open in the Zariski topology on $H_1$.\\ \end{corollary}

\begin{proof} For any $h \in H_1$, by the definition of $H_1$, we have $L' \subset C_L(h)$. Also, as $H$ is Abelian, $H_1 \subset C_L(h)$. This means that $L' +H_1 \subset C_L(h)$ for all $h \in H_1$, therefore\\
\begin{equation*} \dim(C_L(h)) \geq \dim(L' + H_1) = \dim(L') + |\Delta|-|\Delta'|,\\ \end{equation*}
for all $h \in H_1$, with equality only if $C_L(h) = L' +H_1$, i.e. only if $h \in \mathfrak{I}$. For $h \notin \mathfrak{I}$ we have\\
\begin{equation*} \dim(C_L(h))> \dim(L') + |\Delta|-|\Delta'|,\\ \end{equation*}
which is a closed condition, i.e. the complement of $\mathfrak{I}$ is a closed set, making $\mathfrak{I}$ open in the Zariski topology on $H_1$.\\ \end{proof}

\begin{corollary}\label{k} Let $h_1,h_2 \in \mathfrak{I}$. If $g \in G$ is such that $g \cdot h_1 = h_2$, then $g \in N_G(L'+H_1)$.\\ \end{corollary}

\begin{proof} Assume that $g \cdot h_1 = h_2$. If $z \in C_L(h_1) = L' + H_1$, then $g \cdot z \in C_L(h_2) = L'+H_1$, proving that $g \in N_G(L'+H_1)$.\\ \end{proof}

Corollary \label{k} introduced new important object: $N_G(L'+H_1)$. We investigate its properties.\\

\begin{lemma}\label{kolekcija} Let $x_1 \in L'$ be nilpotent and $h \in \mathfrak{I}$. Then the following holds.\\
(1) $Z(L'+H_1) = H_1$\\
(2) $N_G(L'+H_1) = N_G(L') \cap N_G(H_1)$\\
(3) $N_L(H_1) = L'+H_1$\\
(4) $N_L(L')\cap N_L(H_1) = L'+H_1$\\
(5) $\dim(N_G(L'+H_1)) = \dim(L')+|\Delta|-\Delta'|$.\\
(6) $\dim(C_L(x_1+h))-\dim(C_{L'}(x_1)) = |\Delta|-|\Delta'|$.\\
(7) $\dim(\mathrm{Stab}_G(x_1)\cap N_G(L'+H_1)) = \dim(C_{L'}(x_1))+|\Delta|-|\Delta'|$.\\
(8) $\mathfrak{Lie}(\mathrm{Stab}_G(L') \cap \mathrm{Stab}_G(H_1)) = H_1$.\\
\end{lemma}

\begin{proof} (1) is obvious.\\
(2) The inclusion $N_G(L')\cap N_G(H_1) \subset N_G(L'+H_1)$ is obvious.\\

Let $g \in N_G(L'+H_1)$ be arbitrary. The action of $g$ preserves $Z(L'+H_1)$, so $g \in N_G(H_1)$ by (1). Let $\alpha \in \Phi'$ be arbitrary. Then $x_{\alpha} \in L'$ is nilpotent and $g \cdot x_{\alpha}$ is nilpotent by Proposition \ref{b}, hence $g \cdot x_{\alpha} \in L'$ for all $\alpha \in \Phi'$. Take $\beta \in \Delta'$, then\\
\begin{equation*} g \cdot h_{\beta} = g \cdot [x_{\beta},x_{-\beta}] = [g \cdot x_{\beta},g \cdot x_{-\beta}] \in L'.\end{equation*}
We proved that $g \cdot x_{\alpha} \in L'$ for all $\alpha \in \Phi$ and $g \cdot h_{\alpha} \in L'$ for all $\alpha \in \Delta'$, so $g \in N_G(L')$, and $N_G(L'+H_1) \subset N_G(L')\cap N_G(H_1)$.\\
(3) We know that $C_L(H_1) = L'+H_1 \subset N_L(H_1)$.\\

To prove the other inclusion, let $h \in \mathfrak{I}$ and $x \in N_L(H_1)$. Write\\
\begin{equation*} x = \sum_{\alpha \in \Delta} a_{\alpha}h_{\alpha}+\sum_{\beta \in \Phi}b_{\beta}x_{\beta}\end{equation*}
for some $a_{\alpha},b_{\beta} \in F$ for $\alpha \in \Delta,\beta \in \Phi$. As $x \in N_L(H_1)$, and $h \in H_1$, the commutator $[h,x]$ belongs to $H_1$, therefore\\
\begin{equation*} [h,x] = \sum_{\beta \in \Phi}b_{\beta}\beta(h)x_{\beta}\end{equation*}
must be $0$, so $x \in C_L(h) = L'+H_1$.\\
(4) Parts (2) and (3) imply that\\
\begin{equation*} L'+H_1 \subset N_L(L'+H_1) = N_L(L')\cap N_L(H_1) \subset N_L(H_1) = L'+H_1,\end{equation*}
hence all containments are equalities.\\
(5) By (2),(4) and Proposition \ref{b}, we have:\\
\begin{equation*} \mathfrak{L}\mathfrak{i}\mathfrak{e}(N_G(L'+H_1)) = L'+H_1\end{equation*}
and as algebraic group and its Lie algebra have the same dimension, (5) follows by Lemma \ref{d}, using the fact that the sum $L' +H_1$ is direct.\\
(6) We show that:\\
\begin{equation*}C_L(x_1) \cap (L'+H_1) = C_{L'}(x_1)+C_{H_1}(x_1)= C_{L'}(x_1)+H_1.\end{equation*}
The last equality holds by the fact that $x_1 \in L'$. The inclusion $C_{L'}(x_1)+H_1\subset C_L(x_1) \cap (L'+H_1)$ is clear. For the other inclusion, let $y \in C_L(x_1)$. If $y \in L'+H_1$, we can write $y = l+h'$, where $l \in L',h' \in H_1$, and we only need to show that $l \in C_{L'}(x_1)$. This is clear, as $[y,x_1]=0$ and $[h',x_1]=0$ by the definition of $H_1$. Now, using Lemma \ref{dodavanje}, Lemma \ref{d}(3), the assumption $C_L(h) = L'+H_1$, and the fact that the sum $C_{L'}(x_1)+H_1$ is direct, we can compute:\\
\begin{equation*} \dim(C_L(x_1+h))-\dim(C_{L'}(x_1))= \end{equation*}
\begin{equation*} = \dim(C_L(x_1) \cap (L'+H_1))-\dim(C_{L'}(x_1))= \dim(H_1) = |\Delta|-|\Delta'|.\end{equation*}
(7) Using (2),(4) and Proposition \ref{b}, we have:\\
\begin{equation*} \mathfrak{L}\mathfrak{i}\mathfrak{e}(\mathrm{Stab}_G(x_1)\cap N_G(L'+H_1)) = C_L(x_1) \cap N_L(L') \cap N_L(H_1) = C_L(x_1) \cap (L'+H_1).\end{equation*}
(7) now follows from Lemma \ref{d} and the equality\\
\begin{equation*} C_L(x_1) \cap (L'+H_1) = C_{L'}(x_1)+H_1\end{equation*}
which was proven in the proof of (6).\\
(8) By Proposition \ref{b} we have\\
\begin{equation*} \mathfrak{Lie}(\mathrm{Stab}_G(L') \cap \mathrm{Stab}_G(H_1)) =  \mathfrak{Lie}(\mathrm{Stab}_G(L')) \cap \mathfrak{Lie}(\mathrm{Stab}_G(H_1)) = C_L(L') \cap C_L(H_1).\end{equation*}
We know that $C_L(H_1) = L' \oplus H_1$ so\\
\begin{equation*} \mathfrak{Lie}(\mathrm{Stab}_G(L') \cap \mathrm{Stab}_G(H_1)) = C_L(L') \cap (L' \oplus H_1).\end{equation*}
As $H_1 \subset C_L(L')$ and $H_1 \subset L'+H_1$, we know that $H_1 \subset C_L(L') \cap (L' \oplus H_1)$. To prove the other inclusion, let $y \in C_L(L') \cap (L' \oplus H_1)$ be arbitrary. Then there are $y' \in L',h_1 \in H_1$ with $y = y'+h_1$. Then, as $y \in C_L(L')$, for all $x \in L'$ we have $[y,x]=0$. Moreover, $[h_1,x]=0$ by the definition of $H_1$, and so $[y',x]=0$ for all $x \in L'$, which means that $y' \in Z(L')$. As $L'$ is a simple Lie algebra, we get $y'= 0$ and so $y = h_1 \in H_1$.\end{proof}

We conclude this subsection with a lemma that follows immediately from the definition of Borel subalgebras. We will use it in the proof of Theorem \ref{maintheorem2}.\\

\begin{lemma}\label{Borel} Let $x_1,x_2,..., x_{m} \in L$ mutually commute. Then there exists a Borel subalgebra of $L$ that contains them.\\
\end{lemma}

\subsection{Tools from algebraic geometry}

In this section we describe main tools from algebraic geometry that will be used in the proofs. In many cases, we need to calculate dimensions of fibers of suitably constructed maps. For easier access later, we recall here what it says:\\ 

\begin{theorem}\label{fdt}[\textbf{Fibre Dimension Theorem}]\cite[Theorem 11.12]{Haris} Let $X$ be a quasi-projective variety and $\pi:X \to \mathbb{P}^n$ a regular map; let $Y$ be the closure of the image. For any $p \in X$, let $X_p = \pi^{-1}(\pi(p)) \subset X$ be the fibre of $\pi$ through $p$, and let $\mu(p) = \dim_p(X_p)$ be the local dimension of $X_p$ at $p$. Then, for any $m$ the locus of points $p \in X$ such that $\dim_p(X_p) \geq m$ is closed in $X$. Moreover, if $X_0 \subset X$ is any irreducible component, $Y_0 \subset Y$ the closure of its image and $\mu$ the minimum value of $\mu(p)$ on $X_0$, then:\\
\begin{equation*} \dim(X_0) = \dim(Y_0)+\mu.\\ \end{equation*}
\end{theorem}

An important tool will also be smooth points, so we recall their definition.\\

\begin{definition} The \textbf{tangent space} to the variety $V \subset \mathbb{A}^n$  at the point $p$ is the affine space which is the zero locus of $J_1(V) = \{dF:F \in J(V)\}$, where $dF = \sum_{i=1}^n \frac{\partial F}{\partial X_i}(p)(X_i-p_i)$ and $J(V)$ is the vanishing ideal of $V$, i.e. the ideal of polynomials that vanish in every point of $V$.\end{definition}

Dimension of the tangent space at the point $p$ is always bigger than or equal to the local dimension of the variety at that point.

\begin{definition} We say that $p \in V$ is a \textbf{smooth point} of the variety $V$ if the local dimension of $V$ at $p$ is equal to the dimension of the tangent space at $p$.\\ \end{definition}

Our goal is to prove that the variety $C_m(L)$ is reducible in most cases. To do that, we need to find at least two components in $C_m(L)$. One component, which is always present, is introduced below and called regular component. In order to prove that $C_m(L)$ is reducible, we will find a smooth point in an irreducible component of the variety which will be different than the regular component. We will compute the dimension of some irreducible subvariety of $C_m(L)$, which is not equal to the regular component and contains a given point $p$, and estimate the dimension of the tangent space to $C_m(L)$ at the point $p$. It will follow that the dimension of that tangent space is not bigger than the dimension of the chosen subvariety of $C_m(L)$. As mentioned above, it can not be lower either, hence the two dimensions are equal, which means that the point is smooth, and the chosen subvariety of $C_m(L)$ containing $p$ is a component.\\

However, it is not easy to compute the tangent space to $C_m(L)$, as the ideal defined by the natural commutativity relations is often not radical, see \cite{Jel} for the case $m \geq 4,L = \mathfrak{gl}_l,l \geq 8$. We note that in small cases the ideal is radical, for example for $m=2,L = \mathfrak{gl}_l, l \leq 4$ \cite{Hreins}. We therefore define an auxiliary space whose dimension will be larger than or equal to the dimension of the tangent space (see Lemma \ref{tangentspace} below).\\

\begin{definition} To a point $(x_{1,0},...,x_{m,0}) \in C_m(L)$ we associate the vector space \\
\begin{equation}\label{tspace}T_{(x_{1,0},...,x_{m,0})}(C_m(L)) = \{(x_1,...,x_m) \in L^m; [x_{i,0},x_j] = [x_{j,0},x_i] (\forall 1 \leq i,j \leq m)\},\\ \end{equation} 
which we call \textbf{$T$-space}.\\ \end{definition}

\textit{Remark:} For readers familiar with schemes, we note that our $T$-space is actually the tangent space to the scheme $Spec(F[X_1,...,X_{m \dim(L)}]/I)$, where $I$ is the ideal defined by natural quadratic polynomials that describe commutativity relations. However, schemes are never needed in the paper, hence we avoid introducing them, so that the paper is accessible to a wider audience.\\

\subsection{Basic results about $C_m(L)$}

It is known that $C_m(L)$ is a variety. We reprove this fact by giving concrete polynomial equations that define it, as we will need those equations in the second part of Theorem on adding diagonals (Theorem \ref{maintheorem2}). Choose a basis $\{a_1,...,a_n\}$ of $L$. The element $[a_{r},a_{l}] \in L$ can be represented via the basis of $L$ as $[a_{r},a_{l}] = \sum_{k = 1}^{n} \gamma_{r,l,k}a_{k}$, where $\gamma_{r,l,k} \in F$ are structure constants. Given $x_r,x_l \in L$, we can uniquely represent $x_{r} = \sum_{i = 1}^{n}\alpha_{r,i}a_{i}$ and $x_{l} = \sum_{j= 1}^{n} \alpha_{l,j}a_{j}$. Now we have\\
\begin{equation*}[x_{r},x_{l}] = \sum_{k = 1}^{r}(\sum_{j = 1}^{n}\sum_{i = 1}^{n} \alpha_{r,i}\alpha_{l,j}\gamma_{i,j,k})a_{k}.\end{equation*}
Since $\{a_{1},...,a_{n}\}$ is a basis of $L$, the condition $[x_{r},x_{l}] = 0$ is equivalent to\\
\begin{equation}\label{jednacine} \sum_{j = 1}^{n}\sum_{i = 1}^{n} \alpha_{r,i}\alpha_{l,j}\gamma_{i,j,k} = 0 \end{equation}\\
for all $1 \leq r,l \leq m$ and all $k \leq n$. These polynomial equations prove $C_{m}(L)$ is an affine variety. Note that a different choice of basis of $L$ would give an isomorphic variety.\\
\\ 

\begin{definition} Let $\mathcal{L}$ be a reductive Lie algebra, let $\mathcal{G}$ be a connected algebraic group such that $\mathfrak{Lie}(\mathcal{G}) = \mathcal{L}$ and let $\mathcal{H}$ be a Cartan subalgebra of $\mathcal{L}$. The \textbf{regular component} of $C_{m}(\mathcal{L})$ is $\mathrm{Reg}_{m,\mathcal{L}} = \overline{\mathcal{G} \cdot (\mathcal{H} \times \mathcal{H} \times ... \times \mathcal{H})}$ ($\mathcal{H}$ being written $m$ times) where overline means Zariski closure.\\ \end{definition}

The regular component of $C_m(\mathcal{L})$ is closely related to regular semisimple elements of $\mathcal{L}$, so we first recall the definition of regular semisimple elements.\\

\begin{definition} An element $x \in \mathcal{L}$ is called \textbf{regular semisimple} if $C_{\mathcal{L}}(x)$ is a Cartan subalgebra of $\mathcal{L}$.\\ \end{definition}

The regular component is indeed a component of $C_m(\mathcal{L})$. This is well-known, but we were unable to find an explicit proof in the literature, so we give the proof for the sake of completeness.\\

\begin{lemma}\label{Reg} For any $m \in \mathbb{N}$ and every reductive Lie algebra $\mathcal{L}$, the set $\mathrm{Reg}_{m,\mathcal{L}}$ is an irreducible component of $C_{m}(\mathcal{L})$, and has dimension\\
\begin{equation*} \dim(\mathrm{Reg}_{m,\mathcal{L}}) = \dim(\mathcal{L}) + (m - 1)\mathrm{rank}(\mathcal{L}).\\ \end{equation*} \end{lemma}

\begin{proof} We first show that $\mathrm{Reg}_{m,\mathcal{L}}$ is irreducible. $\mathrm{Reg}_{m,\mathcal{L}}$ is the closure of the image of a morphism from an irreducible variety\\
\begin{equation*}\mu : \mathcal{G} \times \mathcal{H} \times \mathcal{H} \times ... \times \mathcal{H} \to \mathrm{Reg}_{m,\mathcal{L}}\end{equation*}
\begin{equation*}(g,x_{1},...,x_{m}) \mapsto (g \cdot x_{1},...,g\cdot x_{m}),\end{equation*}
and as such is irreducible. \\

The set of all regular semisimple elements is open and dense in $\mathcal{L}$ in the Zariski topology (see for example \cite[Corollary 2.1.13]{Col} and references therein). The set of commuting tuples of regular semisimple elements is therefore Zariski open in $C_{m}(\mathcal{L})$ hence its closure is a union of components of $C_m(\mathcal{L})$. To show that $\mathrm{Reg}_{m,\mathcal{L}}$ is a component, it therefore suffices to prove that any commuting tuple of regular semisimple elements $(x_{1},...,x_{m})$  belongs to $\mathcal{G} \cdot (\mathcal{H}\times \mathcal{H}\times ... \times \mathcal{H})$.\\

Since $x_{1}$ is a regular semisimple element, its centralizer $C_{\mathcal{L}}(x_{1})$ is a Cartan subalgebra. All Cartan subalgebras are conjugate, i.e. there exists $g \in \mathcal{G}$ such that $g \cdot \mathcal{H} = C_{\mathcal{L}}(x_{1})$. So, then $x_{i} \in g \cdot \mathcal{H}$ for all $1\leq i\leq m$. In other words, $(x_{1},...,x_{m}) \in \mathcal{G} \cdot (\mathcal{H}\times \mathcal{H} \times ... \times \mathcal{H})$,  and this finishes the proof that $\mathrm{Reg}_{m,\mathcal{L}}$ is indeed a component.\\

We now prove that $\dim(\mathrm{Reg}_{m,\mathcal{L}}) = \dim(\mathcal{G}) + (m - 1)\mathrm{rank}(\mathcal{L})$. For this we need Fibre Dimension Theorem (Theorem \ref{fdt}). First, recall the mapping $\mu$. We know that $\mathrm{Reg}_{m,\mathcal{L}}$ is the closure of the image of this map. All that remains is to calculate the dimension of the fibres on the open set where the first element in the $m$-tuple is regular semisimple. For this, we consider:\\
\begin{equation}\label{fibra}(g \cdot x_{1},...,g\cdot x_{m}) = (g' \cdot x'_{1},...,g' \cdot x'_{m}). \end{equation}\\\

We first consider the first components of the $m$-tuples in (\ref{fibra}). The equation $g \cdot x_{1} = g' \cdot x'_{1}$ is equivalent to $(g')^{- 1}g \cdot x_{1} = x'_{1}$. It follows that $(g')^{- 1}g \cdot C_{\mathcal{L}}(x_{1}) = C_{\mathcal{L}}(x_{1}')$. As $x_{1}, x_{1}' \in \mathcal{H}$ are regular semisimple, their centralizers are Cartan subalgebras. Every two elements of $\mathcal{H}$ commute, so $\mathcal{H} \subset C_{\mathcal{L}}(x_1)$ and $\mathcal{H} \subset C_{\mathcal{L}}(x_1')$. But then, as dimensions of any two Cartan subalgebras are equal, we have $C_{\mathcal{L}}(x_1) = C_{\mathcal{L}}(x_1') = \mathcal{H}$. So, we have $(g')^{-1}g \cdot \mathcal{H} = \mathcal{H}$, or $(g')^{- 1}g \in N_{\mathcal{G}}(\mathcal{H})$. Conversely, if $(g')^{-1}g \in N_{\mathcal{G}}(\mathcal{H})$ and $x_1' = (g')^{-1}g \cdot x_1$, then it is clear that $x_1' \in \mathcal{H}$ and $g \cdot x_1 = g' \cdot x_1'$.\\
 
This finishes our analysis of the first equation from (\ref{fibra}) . We consider all other $m- 1$ equations at once: for every $2 \leq i \leq m$ we have $g \cdot x_{i} = g' \cdot x'_{i}$. Then $(g')^{- 1}g \cdot x_{i} = x'_{i}$. Conversely, $x_i'= (g')^{-1}g \cdot x_i$ indeed belongs to $\mathcal{H}$, as $(g')^{-1}g \in N_{\mathcal{G}}(\mathcal{H})$. This proves that all the $x'_{i}$ for $2\leq i \leq m$ are uniquely defined, once $x_{i}$, $g,g'$ are fixed,  and hence bring no additional dimension. More precisely,\\
\begin{equation*} \mu^{- 1}(g \cdot x_{1},...,g \cdot x_{m}) = \{(gP^{- 1}, Px_{1},...,Px_{m}): P \in N_{\mathcal{G}}(\mathcal{H})\} .\end{equation*}\\

By Proposition \ref{b} we have $\mathfrak{Lie}(N_{\mathcal{G}}(\mathcal{H})) = N_{\mathcal{L}}(\mathcal{H})$ which is equal to $\mathcal{H}$, as $\mathcal{H}$ is a Cartan subalgebra of $\mathcal{L}$. This proves that $N_{\mathcal{G}}(\mathcal{H})$ has the same dimension as $\mathcal{H}$, i.e. $\mathrm{rank}(\mathcal{L})$, hence the dimension of the fibre is $\mathrm{rank}(\mathcal{L})$. By the Fibre Dimension Theorem, we have that\\
\begin{equation*} \dim(\mathrm{Reg}_{m,\mathcal{L}}) = \dim(\mathcal{G}) + m\ \cdot\ \dim(\mathcal{H}) - \mathrm{rank}(\mathcal{L}) = \dim(\mathcal{G}) + (m - 1)\mathrm{rank}(\mathcal{L}).\end{equation*} \end{proof}
 
\begin{lemma}\label{tangentspace} Let $L$ be a simple Lie algebra. The tangent space to $C_m(L)$ at $(x_{1,0},...,x_{m,0})$ is a subspace of the $T$-space associated to $(x_{1,0},...,x_{m,0})$.\\
\end{lemma}

\begin{proof} Let $I$ be the ideal generated by the polynomials from (2). We will prove that the $T$-space is equal to $V(\{dF:F \in I\})$. The lemma then follows from the fact that $I \subset \sqrt{I}$ and that $\sqrt{I}$ is the vanishing ideal of $C_m(L)$.\\

We compute the differentials of the generators of $I$. Let $\{a_1,...,a_n\}$ be the basis of $L$ chosen at the beginning of this subsection, let $r,l \in \{1,...,m\}$, $r \neq l$ be arbitrary. Then we can put\\
\begin{equation*} x_r = \sum_{i=1}^n \alpha_{r,i} a_i,\ x_l = \sum_{j = 1}^{n}\alpha_{l,j}a_{j},\ \mathrm{and}\ x_{r,0} = \sum_{i=1}^n \alpha_{r,i,0} a_i,\ x_{l,0} = \sum_{j=1}^n \alpha_{l,j,0} a_j.\end{equation*}

For every $r,l \in \{1,...,m\}$ and $k \leq n$ let:\\
\begin{equation*} f_{r,l,k}(\alpha_{r, 1},...,\alpha_{r,n},\alpha_{l,1},...,\alpha_{l,n}) = \sum_{i=1}^n \sum_{j=1}^n \alpha_{r,i} \alpha_{l,j} \gamma_{i,j,k}.\end{equation*}\\
Note that polynomials $f_{r,l,k}$ generate $I$. The linear part of this polynomial at the point $(\alpha_{r,1,0},...,\alpha_{l,n,0})$ is equal to\\
\begin{equation*}d_{(x_{1,0},...,x_{m,0})}f_{r,l,k} = \sum_{i=1}^n \sum_{j=1}^n \alpha_{l, j,0} (\alpha_{r,i} -\alpha_{r, i,0})\gamma_{i,j,k}+ \sum_{i=1}^n\sum_{j=1}^n \alpha_{r, i,0} (\alpha_{l,j} - \alpha_{l, j,0}) \gamma_{i,j,k},\end{equation*}\\
and the condition for the tangent space implies that this is zero.\\

This happens for all $1 \leq k \leq n$, so when we multiply the equation by $a_k$ and sum over $k$ we get an equivalent equation\\
\begin{equation*} \sum_{k=1}^n\left(\sum_{i=1}^n \sum_{j=1}^n \alpha_{l,j,0} (\alpha_{r,i} -\alpha_{r, i,0})+ \sum_{i=1}^n\sum_{j=1}^n \alpha_{r,i,0} (\alpha_{l,j} - \alpha_{l, j,0})\right) \gamma_{i,j,k} a_k = 0. \end{equation*}\\
Recalling the definition of the structure constants we can rewrite the last sum as\\
\begin{equation*}\left[\sum_{i=1}^n\alpha_{r,i}a_i, \sum_{j=1}^n\alpha_{l,j,0}a_j\right] - 2 \left[\sum_{i=1}^n \alpha_{r,i,0}a_i,\sum_{j=1}^n \alpha_{l,j,0}a_j\right] + \left[\sum_{i=1}^n\alpha_{r,i,0}a_i,\sum_{j=1}^n\alpha_{l,j} a_j\right] = 0, \end{equation*}\\
which is equivalent to $[x_r,x_{l,0}]-2[x_{r,0},x_{l,0}] + [x_{r,0},x_l] = 0.$\\

The point $(x_{1,0},x_{2,0},...,x_{m,0})$ lies on the variety $C_{m}(L)$ hence $[x_{r,0},x_{l,0}]=0$ and we get $[x_r,x_{l,0}]+ [x_{r,0},x_l] = 0$ for $r,l \in \{1,2,...,m\}$, which are the equations for $T$- space.\\ \end{proof}

\section{Theorem on Adding Diagonals}

The aim of this section is to prove a theorem, which says that if $C_{m}(L')$ is reducible for some simple Lie subalgebra $L'$ of $L$ whose Dynkin diagram is a subdiagram of the Dynkin diagram of $L$, then in many cases $C_{m}(L)$ is reducible.\\

We introduce the following notation. For $x_1' \in L$ we denote:\\
\begin{equation*} x_1' +H_1 = \{x_1'+h: h \in H_1\}.\end{equation*}
Here is the formulation of the first part of the main theorem of this section:\\
\\
\begin{theorem}\label{maintheorem1}[\textbf{Theorem on adding diagonals, first part}] Assume the notation given in Section 2.1 for $L,L',G,H,H',H_1,\Phi,\Phi',\Delta, \Delta'$, and let $G'$ be a connected simple algebraic subgroup of $G$ such that $\mathfrak{Lie}(G')=L'$. Assume that $x_1' \in L'$ is nilpotent and $S'$ is a $\mathrm{Stab}_{G'}(x_{1}')^{\circ}$-stable irreducible subvariety of $C_{m - 1}(C_{L'}(x_{1}'))$ containing only $(m-1)$-tuples of nilpotent elements. Let\\
\begin{equation*} C' = G' \cdot (\{x_1'\} \times S') \subset C_m(L'). \end{equation*}\\
Furthermore, let\\
\begin{equation*} S = \{(x_2' + h_{2},..,x_{m}' + h_m): (x_2',...,x_{m}') \in S',\ h_2, ..., h_{m} \in H_{1}\} \subset C_{m-1}(L)\end{equation*}\\
and\\
\begin{equation*}C = G \cdot ((x_1'+H_1) \times S) \subset C_m(L). \end{equation*}\\
Then\\
\begin{equation*}\dim(C) = \dim(C')+\dim(G) - \dim(G') +(m - 1)(|\Delta| - |\Delta'|).\end{equation*}\\
\end{theorem}

Before we prove this theorem, we prove a lemma we need in the proof. This lemma is known (see for example \cite[proof of Corollary 4.2]{Ngo}), but we prove it for the sake of completeness.\\

\begin{lemma}\label{Lemma11} Under the assumptions of the theorem the following holds:\\
\begin{equation*} \dim(C') = \dim(G') + \dim(S') - \dim(C_{L'}(x_{1}')).\end{equation*}\\
\end{lemma}

\begin{proof} We consider the surjective mapping\\
$$f: G' \times S' \to C'$$\\
$$(g,x_{2}',...,x_{m}') \mapsto (g \cdot x_{1}', g \cdot x_{2}', ..., g \cdot x_{m}').$$\\
We first consider the fibres of this map. Denote\\
\begin{equation*} \mathcal{F} = f^{-1}(g \cdot x_1', g\cdot x_2',...,g \cdot x_m'),\end{equation*}
\begin{equation*} \mathcal{A} = \{(g',(g')^{-1}g \cdot x_2',...,(g')^{-1}g \cdot x_m'): (g')^{-1}g \in (\mathrm{Stab}_{G'}(x_1'))^{\circ}\},\end{equation*}
\begin{equation*} \mathcal{B} = \{(g', (g')^{- 1}g \cdot x_{2}',...,(g')^{- 1}g \cdot x_{m}'): (g')^{- 1}g \in \mathrm{Stab}_{G'}(x_{1}')\}.\end{equation*}
From\\
\begin{equation*} (g \cdot x_{1}', g\cdot x_{2}',..., g\cdot x_{m}') = (g' \cdot x_{1}', g' \cdot x_{2}'', ..., g' \cdot x_{m}'') \end{equation*}\\
we first get $g \cdot x_{1}' = g' \cdot x_{1}'$ so $(g')^{- 1}g \in \mathrm{Stab}_{G'}(x_{1}').$ Next, we get $x_{i}'' = (g')^{- 1}g \cdot x_{i}'$ for $2 \leq i \leq m$, so $\mathcal{F} \subset \mathcal{B}$. Next, suppose that $(g')^{- 1}g \in \mathrm{Stab}_{G'}(x_{1}')^{\circ}$ and $(x_2',...,x_m') \in S'$. Then\\
\begin{equation*}[x_{1}', (g')^{- 1}g \cdot x_{i}'] = [(g')^{- 1}g \cdot x_{1}', (g')^{- 1}g \cdot x_{i}'] = 0, \end{equation*}\\
i.e. $(g')^{-1}g \cdot x_i' \in C_{L'}(x_1')$ for every $2 \leq i \leq m$. Now, we recall that $S'$ is $\mathrm{Stab}_{G'}(x_1')^{\circ}$-stable to conclude that if $x_{i}'' = (g')^{- 1}g \cdot x_{i}'$, for all $2 \leq i \leq m$, then $(x_{2}'',...,x_{m}'') \in S'$, hence $\mathcal{A} \subset \mathcal{F}$.\\

If we put $g'= g$ in the definition of $\mathcal{A}$, it becomes obvious that $(g,x_2',...,x_m') \in \mathcal{A}$. Then\\
\begin{equation*} \dim(\mathrm{Stab}_{G'}(x_1')^{\circ}) = \dim_{(g,x_2',...,x_m')}\mathcal{A} \leq \dim_{(g,x_2',...,x_m')}\mathcal{F} \leq \dim(\mathcal{B}) = \dim(\mathrm{Stab}_{G'}(x_1')).\end{equation*}
It follows that\\
\begin{equation*} \dim_{(g,x_2',...,x_m')}\mathcal{F} = \dim(C_{L'}(x_1')).\end{equation*}
Finally, by Fibre Dimension Theorem (Theorem \ref{fdt}), we get\\
$$ \dim(C') = \dim(G') + \dim(S') - \dim(C_{L'}(x_{1}')).$$ \end{proof}

\begin{proof}[Proof of Theorem \ref{maintheorem1}]: We consider the mapping:\\
$$f: G \times H_{1} \times S \to C$$\\
$$(g, h, x_{2},...,x_{m}) \mapsto (g \cdot(x_{1}' + h), g \cdot x_{2},..., g\cdot x_{m}).$$\\
The map $f$ is clearly surjective. We want to use Fibre Dimension Theorem for this mapping. For that we first need to calculate the fibres. We will only calculate the fibres for points for which $h \in \mathfrak{I}$. This is sufficient, since $\mathfrak{I}$ is open in the Zariski topology by Corollary \ref{i}. We will not calculate the fibres precisely, rather we will prove that fibres contain a set $A$ and are contained in a set $B$, and then we will show that sets $A$ and $B$ are of the same dimension. Denote:\\
\begin{equation*} \mathcal{F} = f^{-1}(g \cdot(x_1'+h),g \cdot x_2,...,g \cdot x_m).\end{equation*}
We begin by analysing when two images are the same:\\
$$(g \cdot (x_{1}' + h), g \cdot x_{2},..., g \cdot x_{m}) = (g'\cdot(x_{1}' + h''), g' \cdot x_{2}'',..., g' \cdot x_{m}'').$$\\
First components of the above $m$-tuples have to be equal. Using the fact that $g \cdot(x_{1}' + h)$ has nilpotent part $g \cdot x_{1}'$ and semisimple part $g \cdot h$ (by Proposition \ref{b}) we get that $g \cdot x_{1}' = g' \cdot x_{1}'$ and $g \cdot h = g' \cdot h''$. The equality $g \cdot x_{1}' = g' \cdot x_{1}'$ is equivalent to  $(g')^{- 1}g \in \mathrm{Stab}_{G}(x_{1}')$ and the equation $g \cdot h = g' \cdot h''$ is equivalent to $(g')^{- 1}g \cdot h = h''$. By Corollary 2.11, we get $(g')^{-1}g \in N_G(L'+H_1)$.\\

We also get $x_i'' =(g')^{-1}g \cdot x_i$ for $2 \leq i \leq m$, so\\
\begin{equation*} \mathcal{F} \subset \{(g',(g')^{-1}g \cdot h, (g')^{-1}g \cdot x_2,...,(g')^{-1}g \cdot x_m): (g')^{-1}g \in \mathrm{Stab}_G(x_1') \cap N_G(H_1+L')\}.\end{equation*}
We denote the set on the right by $B$.\\

The dimension of $B$ is obviously equal to the dimension of $\mathrm{Stab}_{G}(x_{1}') \bigcap N_{G}(L'+H_{1})$, which is equal to\\
\begin{equation*} \dim(C_{L'}(x_1')) + |\Delta| - |\Delta'|,\\ \end{equation*}
by Lemma \ref{kolekcija}(7).\\

Next, we want to prove that\\
\begin{equation}\label{donja} A \subset \mathcal{F}\end{equation}
where\\
\begin{equation*} A = \{(g',(g')^{-1}g \cdot h, (g')^{-1}g \cdot x_2,...,(g')^{-1}g \cdot x_m): (g')^{-1}g \in <\mathrm{Stab}_{G'}(x_1')^{\circ},(\mathrm{Stab}_G(L')\cap \mathrm{Stab}_G(H_1))^{\circ}>\}.\end{equation*}

Let $(g',(g')^{-1}g \cdot h, (g')^{-1}g \cdot x_2,..., (g')^{-1}g \cdot x_m) \in A$ be arbitrary. By the definition of $f$ it is clear that we have to show only that $(g')^{-1}g(H_1) \subset H_1$, that $(g')^{-1}g \cdot x_1'= x_1'$ and that $((g')^{-1}g \cdot x_2,...,(g')^{-1}g \cdot x_m) \in S$. To prove this it is enough to consider the case when $(g')^{-1}g \in \mathrm{Stab}_{G'}(x_1')^{\circ}$ or $(g')^{-1}g \in (\mathrm{Stab}_G(L') \cap \mathrm{Stab}_G(H_1))^{\circ}$. In the second case $(g')^{-1}g$ fixes $L'$ and $H_1$, so the conclusion is obvious. Suppose $(g')^{-1}g \in \mathrm{Stab}_{G'}(x_1')^{\circ}$. As $\mathfrak{Lie}(G') = L'$ and $\mathfrak{Lie}(\mathrm{Stab}_G(H_1)^{\circ}) = C_L(H_1) = L'+H_1$, by Proposition 2.1(6) it follows that $G' \subset \mathrm{Stab}_G(H_1)$ hence $\mathrm{Stab}_{G'}(x_1')^{\circ} \subset \mathrm{Stab}_G(H_1)$. It follows that $(g')^{-1}g(H_1) \subset H_1$ while the equality $(g')^{-1}g \cdot x_1' = x_1'$ is obvious. Finally, as $S'$ is $\mathrm{Stab}_{G'}(x_1')^{\circ}$ -stable, we get $((g')^{-1}g \cdot x_2,..., (g')^{-1}g \cdot x_m) \in S$. This finishes the proof that $A \subset \mathcal{F}$.\\
  
Taking $g' = g$ we see that the point $(g,h,x_{2},..,x_{m})$ belongs to $A$, so the local dimension of the fibre at $(g,h,x_2,...,x_m)$ is larger than or equal to the local dimension of $A$ at that point.\\

It is clear that $\dim_{(g,h,x_2,...,x_m)}A =  \dim<\mathrm{Stab}_{G'}(x_{1}')^{\circ},(\mathrm{Stab}_G(L') \cap \mathrm{Stab}_G(H_1))^{\circ}>$ which is equal to $\dim(C_{L'}(x_{1}')+ H_1)$ (see Proposition \ref{b}(7), Lemma \ref{kolekcija}(8) and note that $C_{L'}(x_1')$ and $H_1$ commute). Hence the local dimension of $A$ at $(g,h,x_2,...,x_m)$ is equal to the dimension of $B$, and to the local dimension of the fibre at $(g,h,x_1,...,x_m)$.\\

Now we use Fibre Dimension Theorem (Theorem \ref{fdt}):\\
\begin{eqnarray*}\dim(C) = \dim(G) + \dim(H_{1}) + \dim(S) - \dim(C_{L'}(x_{1}') + H_{1})=\\
= \dim(G) + \dim(S) - \dim(C_{L'}(x_{1}')).
\end{eqnarray*}\\
By Lemma \ref{Lemma11} we have\\
$$\dim(C') = \dim(G') + \dim(S') - \dim(C_{L'}(x_{1}'))$$\\
and finally we get\\
\begin{eqnarray*} \dim(C) - \dim(C') = \dim(G) - \dim(G') + \dim(S) - \dim(S') =\\
= \dim(G) - \dim(G')+(m-1)\dim(H_1)=\\
= \dim(G) - \dim(G') + (m- 1)(|\Delta| - |\Delta'|).
\end{eqnarray*}\\ \end{proof}

\begin{corollary} Suppose that there exists a simple Lie subalgebra $L'$ of $L$ such that the Dynkin diagram of $L'$ is a subdiagram of the Dynkin diagram of $L$, and such that there exists a component of $C_m(L')$ of dimension at least $\dim(L')+(m-1)\mathrm{rank}(L')$, different from $\mathrm{Reg}_{m,L'}$. Then, $C_m(L)$ is reducible.\\ \end{corollary}

\begin{proof}  Let $L'$ be a Lie subalgebra of $L$ of the smallest dimension that satisfies the conditions of the corollary. Let $G'$ be a connected simple algebraic subgroup of $G$ such that $\mathfrak{Lie}(G') = L'$. Let $H$ be a Cartan subalgebra of $L$, let $\Phi$ be root system that corresponds to the pair $(L,H)$ and let $\Delta$ be its base. Let $\Delta' \subset \Delta$ correspond to the vertices of the Dynkin diagram of $L'$ and  let $\Phi' = \Phi \cap \mathrm{span}(\Delta')$. Note that $\Phi'$ is irreducible, as $L'$ is simple and its Dynkin diagram is connected. Consider a Chevalley basis of $L$ and define $H' = \mathrm{span}(h_{\alpha}:\alpha \in \Delta')$. Then $H'+\sum_{\alpha \in \Phi'}L_{\alpha}$ is a simple Lie algebra with the same Dynkin diagram as $L'$. As isomorphic Lie algebras yield isomorphic commuting varieties, we may assume that $L'=H'+\sum_{\alpha \in \Phi'}L_{\alpha}$.\\
  
In $C_m(L')$ there is at least one component of dimension at least $\dim(\mathrm{Reg}_{m,L'})$, which is different from $\mathrm{Reg}_{m,L'}$, and we fix one such, and denote it by $\mathfrak{C}$. We first want to prove that $\mathfrak{C}$ contains only nilpotent $m$-tuples. This part of the proof is inspired by proof of \cite[Corollary 4.9]{Ngo}. Suppose that there exists an $m$-tuple $(x_1,...,x_m) \in \mathfrak{C}$ which has one element not nilpotent. Without loss of generality, we may assume that $x_1$ is not nilpotent. It has nonzero semisimple part that we denote by $x_s$. The centralizer $C_{L'}(x_s)$ is reductive Lie algebra by \cite[Lemma 2.1.2]{Col} and it contains $x_1,...,x_m$. Let $\mathfrak{C}_{\mathcal{S}}$ be the set of all nonzero semisimple elements of $L'$ that appear as semisimple parts of elements of $m$-tuples from $\mathfrak{C}$. This set is nonempty by the assumption. Since the set of $m$-tuples containing a non-nilpotent element is open, and $\mathfrak{C}$ is irreducible, the above argument shows that\\
\begin{equation*} \mathfrak{C} \subset \cup_{x \in \mathfrak{C}_{\mathcal{S}}}\overline{(G' \cdot C_m(C_{L'}(x)))}.\end{equation*}
From \cite[Lemma 2.1.2]{Col} it immediately follows that this union is in fact finite, hence by irreducibility of $\mathfrak{C}$ there exists some $x \in \mathfrak{C}_{\mathcal{S}}$ such that\\
\begin{equation*} \mathfrak{C} = \overline{\mathfrak{C} \cap G' \cdot C_m(C_{L'}(x))}.\end{equation*}
Algebra $C_{L'}(x)$ is reductive, hence we can represent\\
\begin{equation*} C_{L'}(x) = Z(C_{L'}(x)) \oplus \oplus_{i=1}^k L'_i\end{equation*}
for some simple Lie algebras $L'_i$. Note that $C_{L'}(x)$ is strictly contained in $L'$, as $x$ is nonzero. Then\\
\begin{equation}\label{triangle} \mathfrak{C} = \overline{G' \cdot (Z(C_{L'}(x))^m \oplus \oplus_{i=1}^kC_m(L'_i))}.\end{equation}
Let $\tilde{\mathfrak{C}}$ be a component of $(Z(C_{L'}(x))^m \oplus \oplus_{i=1}^k C_m(L_i'))$ such that\\
\begin{equation*} \mathfrak{C} = \overline{G' \cdot \tilde{\mathfrak{C}}}.\end{equation*}
By minimality of $L'$ we have\\
\begin{equation*} \dim(\tilde{\mathfrak{C}}) \leq \dim(Z(C_{L'}(x))^m \oplus \oplus_{i=1}^k C_m(L'_i)) \leq \end{equation*}
\begin{equation*} \leq m\dim(Z(C_{L'}(x)))+ \sum_{i=1}^k(\dim(L_i')+(m-1)\mathrm{rank}(L_i')) = \end{equation*}
\begin{equation*} = \dim(C_{L'}(x))+(m-1)\mathrm{rank}(C_{L'}(x)) =\end{equation*} 
\begin{equation*} = \dim(\mathrm{Reg}_{m,C_{L'}(x)}).\end{equation*}
If $\dim(\tilde{\mathfrak{C}}) = \dim(\mathrm{Reg}_{m,C_{L'}(x)})$, then $\dim(C_m(L_i'))=\dim(L_i')+(m-1)\mathrm{rank}(L_i')$ for all $i$, so by minimality of $L'$ we have $\tilde{\mathfrak{C}} = \mathrm{Reg}_{m,C_{L'}(x)}$. From (\ref{triangle}) we can conclude that $\mathfrak{C} \subset \mathrm{Reg}_{m,L'}$ and so $\mathfrak{C}= \mathrm{Reg}_{m,L'}$, but this is contradictory to our assumption. Hence we have\\
\begin{equation}\label{maltese} \dim(\tilde{\mathfrak{C}}) < \dim(\mathrm{Reg}_{m,C_{L'}(x)}) = \dim(C_{L'}(x))+(m-1)\mathrm{rank}(C_{L'}(x)).\end{equation}
Consider the dominant mapping\\
\begin{equation*} f: G' \times \tilde{\mathfrak{C}} \to \mathfrak{C}\end{equation*}
\begin{equation*} (g,(x_1,...,x_m)) \mapsto (g \cdot x_1,...,g \cdot x_m).\end{equation*}
For every $(g,(x_1,...,x_m)) \in G' \times \tilde{\mathfrak{C}}$ we have\\
\begin{equation*} f^{-1}(g \cdot x_1,...,g \cdot x_m) \supset \{(g',x_1',...,x_m'):(g')^{-1}g \in N_{G'}(C_{L'}(x))^{\circ}, x_i' = (g')^{-1}g \cdot x_i, \forall 1 \leq i \leq m\},\end{equation*}
and denote the set on the right by $M$. Note that the connected group $N_{G'}(C_{L'}(x))^{\circ}$ acts on $C_m(C_{L'}(x))$, so all irreducible components of $C_m(C_{L'}(x))$ are invariant under this action, so $M$ is indeed a subset of $G' \times \tilde{\mathfrak{C}}$. We know that\\
\begin{equation*}\dim(M) = \dim(N_{G'}(C_{L'}(x))) = \dim(N_{L'}(C_{L'}(x))) \geq \dim(C_{L'}(x)).\end{equation*}
As $f^{-1}(g \cdot x_1,...,g \cdot x_m) \supset M$, we have\\
\begin{equation*} \dim(f^{-1}(g \cdot x_1,...,g \cdot x_m)) \geq \dim(M) \geq \dim(C_{L'}(x)).\end{equation*}
On the other side, by applying Fibre Dimension Theorem for $f$ and by (\ref{maltese}) we have\\
\begin{equation*}\dim(\mathfrak{C}) = \dim(G' \cdot \tilde{\mathfrak{C}}) = \dim(G')+\dim(\tilde{\mathfrak{C}})-\min(\dim(f^{-1}(g \cdot x_1,...,g \cdot x_m))) <\end{equation*}
\begin{equation*} < \dim(G')+\dim(C_{L'}(x))+(m-1)\mathrm{rank}(C_{L'}(x))-\dim(C_{L'}(x)) \leq \end{equation*}
\begin{equation*} \leq \dim(G')+(m-1) \mathrm{rank}(L')\end{equation*}
which is direct contradiction with the choice of $\mathfrak{C}$. Hence we proved that $\mathfrak{C}$ contains only nilpotent $m$-tuples.\\

For all nilpotent $x_1 \in L'$, we define\\
\begin{equation*} S_{x_1} = \{(x_2,...,x_m): (x_1,x_2,...,x_m) \in \mathfrak{C}\}.\end{equation*}
Then\\
\begin{equation*} \mathfrak{C}= \cup_{x_1} \overline{G' \cdot (\{x_1\} \times S_{x_1})}\end{equation*}
and this union is finite since $L'$ has finitely many nilpotent orbits. Therefore, there exists $x_1$ such that\\
\begin{equation*} \mathfrak{C} = \overline{G' \cdot (\{x_1\} \times S_{x_1})}.\end{equation*}
$S_{x_1} \subset C_{m-1}(C_{L'}(x_1))$ and contains only nilpotent elements. As $\mathfrak{C}$ is irreducible, there exists a component $\tilde{S}$ of $S_{x_1}$ such that\\
\begin{equation*} \mathfrak{C} = \overline{G' \cdot (\{x_1\} \times \tilde{S})}.\end{equation*}
As $\mathfrak{C}$ is a component, conjugating with $\mathrm{Stab}_{G'}(x_1')^{\circ}$ preserves $S_{x_1}$, and hence it preserves its irreducible component $\tilde{S}$. Now use Theorem \ref{maintheorem1} and note that the closure of the set $C$ obtained in Theorem 3.1 is not the regular component of $C_m(L)$.\end{proof}

For the next theorem we recall that the $T$-space associated to the point  $(x_{1,0},...,x_{m,0}) \in C_m(L)$ is given by (\ref{tspace}) and denoted by $T_{(x_{1,0},...,x_{m,0})}(C_m(L))$.\\

\begin{theorem}\label{maintheorem2}[\textbf{Theorem on adding diagonals, second part}] Assume that in Theorem \ref{maintheorem1} there additionally exists $(x_{2,0}',...,x_{m,0}') \in S'$ such that $\dim(C') = \dim(T_{(x_1',x_{2,0}',...,x_{m,0}')}C_m(L'))$, and let $h \in \mathfrak{I}$. Then\\
\begin{equation*} \dim(T_{(x_1'+h,x_{2,0}',...,x_{m,0}')}C_m(L))=\dim(C')+\dim(G)-\dim(G')+(m-1)(|\Delta|-|\Delta'|),\end{equation*}
and $(x_1'+h,x_{2,0}',...,x_{m,0}')$ is a smooth point of $C$, which is a non-regular component of $C_m(L)$. \end{theorem}

\begin{proof} We want to prove that\\
\begin{equation}\label{t3pprva} \dim(T_{(x_1'+h,x_{2,0}',...,x_{m,0}')}C_m(L)) \leq \dim(T_{(x_1',x_{2,0}',...,x_{m,0}')}C_m(L'))+\dim(G)-\dim(G')+(m-1) (|\Delta|-|\Delta'|).\end{equation}
Suppose this is true. As $C \subset C_m(L)$ and by Lemma 2.21 we have\\
\begin{equation}\label{t3druga} \dim(C) \leq \dim_{(x_1'+h,x_{2,0}',...,x_{m,0}')}C_m(L) \leq \dim(T_{(x_1'+h,x_{2,0}',...,x_{m,0}')}C_m(L)),\end{equation}
so by (\ref{t3pprva}) we have\\
\begin{equation}\label{t3treca} \dim(C) \leq \dim(T_{(x_1',x_{2,0}',...,x_{m,0}')}C_m(L'))+\dim(G)-\dim(G')+(m-1)(|\Delta|-|\Delta'|).\end{equation}
The assumption of our theorem then implies that\\
\begin{equation}\label{t3cetvrta} \dim(C) \leq \dim(C')+\dim(G)-\dim(G')+(m-1)(|\Delta|-|\Delta'|).\end{equation}
But by Theorem \ref{maintheorem1} the inequality (\ref{t3cetvrta}) is actually an equality, hence (\ref{t3pprva})-(\ref{t3treca}) are equalities, ending the proof.\\

Now we prove (\ref{t3pprva}). Because of Lemma \ref{Borel}, and since all Borel subalgebras are conjugate, we can assume that $x_1',x_{2,0}',...,x_{m,0}'$ are linear combinations of elements of Chevalley basis that correspond to positive roots. We are searching for $T$-space, i.e. for $(z_1,...,z_m) \in L^m$ such that:\\
\begin{equation}\label{t3peta}[x_1'+h,z_i] = [x_{i,0}',z_1]\ (\forall\ 2 \leq i \leq m)\end{equation}
and\\
\begin{equation}\label{t3sesta}[x_{i,0}',z_j]=[x_{j,0}',z_i]\ (\forall\ 2 \leq i,j \leq m).\end{equation}
For  all $1 \leq i \leq m$, we represent:\\
\begin{equation}\label{t3sedma} z_i = h_i'+h_i''+z_i'+z_i''\end{equation}
where $h_i' \in H',h_i'' \in H_1,z_i' \in \oplus_{\alpha \in \Phi'}L_{\alpha},z_i'' \in \oplus_{\alpha \in \Phi \setminus \Phi'}L_{\alpha}$. Using the fact that $H_1 = C_H(L')$, the equations (\ref{t3peta}) and (\ref{t3sesta}) are equivalent to\\
\begin{equation}\label{t3osma}[x_1',h_i']+[x_1',z_i']+[x_1',z_i'']+[h,z_i'']=[x_{i,0}',h_1']+[x_{i,0}',z_1']+[x_{i,0}',z_1'']\end{equation}
and\\
\begin{equation}\label{t3deveta}[x_{i,0}',h_j']+[x_{i,0}',z_j']+[x_{i,0}',z_j'']=[x_{j,0}',h_i']+[x_{j,0}',z_i']+[x_{j,0}',z_i''].\end{equation}
The inequality (\ref{t3pprva}) that we are proving estimates the difference between dimensions of the $T$-space to $C_m(L)$ associated to the point $(x_1'+h,x_{2,0}',...,x_{m,0}')$ and the $T$-space to $C_m(L')$ associated to the point $(x_1',x_{2,0}',...,x_{m,0}')$. If $(z_1'+h_1',...,z_m'+h_m') \in (L')^m$ are as above, the equations of the $T$-space to $C_m(L')$ associated to the point $(x_1',x_{2,0}',...,x_{m,0}')$ are\\
\begin{equation}\label{t3deseta}[x_1',z_i'+h_i']=[x_{i,0}',z_1'+h_1'],\ 2\leq i\leq m,\end{equation}
\begin{equation}\label{t3jedanaesta}[x_{i,0}',z_j'+h_j']=[x_{j,0}',z_i'+h_i'],\ 2 \leq i,j \leq m.\end{equation}
Note that $T_{(x_1+h,x_{2,0}',...,x_{m,0}')}C_m(L)$ is a direct sum of $T_{(x_1,x_{2,0}',...,x_{m,0}')}C_m(L')$ and the vector space $V$ of all $m$-tuples $(h_1'' + z_1'',...,h_m''+z_m'') \in (H_1 \oplus \oplus_{\alpha \in \Phi \setminus \Phi'}L_{\alpha})^m$ which satisfy\\
\begin{equation}\label{t3dvanaesta}[x_1',z_i'']+[h,z_i'']=[x_{i,0}',z_1''],\ 2\leq i \leq m \end{equation}
and
\begin{equation}\label{t3trinaesta}[x_{i,0}',z_j'']=[x_{j,0}',z_i''],\ 2 \leq i,j \leq m.\end{equation}
The difference between dimensions of $T$-spaces is therefore equal to the dimension of $V$. We analyse equation (\ref{t3dvanaesta}). Recalling that $x_1',x_{2,0}',...,x_{m,0}' \in \oplus_{\beta \in (\Phi')^+}L_{\beta}$ we represent:\\
\begin{equation*} x_1' = \sum_{\beta \in (\Phi')^+}b_{\beta,1}x_{\beta}, \end{equation*}
\begin{equation*} z_i'' = \sum_{\alpha \in \Phi \setminus \Phi'}a_{\alpha,i}x_{\alpha},\ 1\leq i \leq m,\end{equation*}
\begin{equation*} x_{i,0}' = \sum_{\beta \in (\Phi')^+}b_{\beta,i}x_{\beta},\ 2 \leq i \leq m.\end{equation*}
Take $\beta \in \Phi',\alpha \in \Phi \setminus \Phi'$. Recall that $[x_{\beta},x_{\alpha}] = R_{\beta,\alpha}x_{\alpha+\beta}$ if $\alpha + \beta \in \Phi$, $[x_{\beta},x_{\alpha}] = 0$ if $\alpha + \beta \notin \Phi$, and that $[h,x_{\alpha}] = \alpha(h)x_{\alpha}$. The equation (\ref{t3dvanaesta}) is equivalent to\\
\begin{equation*}\sum_{\alpha \in \Phi \setminus \Phi'}\left(\sum_{\beta \in (\Phi')^+\atop \alpha+\beta \in \Phi}(a_{\alpha,i}b_{\beta,1}-a_{\alpha,1} b_{\beta,i})R_{\beta,\alpha}x_{\alpha+\beta}+a_{\alpha,i}\alpha(h)x_{\alpha}\right) = 0.\end{equation*}
In the last equation we do relabeling and write it as\\
\begin{equation*}\sum_{\beta_1 \in \Phi \setminus \Phi'}\left(\sum_{\beta \in (\Phi')^+\atop \beta_1-\beta \in \Phi \setminus \Phi'}(a_{\beta_1-\beta,i}b_{\beta,1}-a_{\beta_1-\beta,1}b_{\beta,i}) R_{\beta,\beta_1-\beta}+a_{\beta_1,i}\beta_1(h)\right)x_{\beta_1} = 0.\end{equation*}
As $x_{\beta_1}$ are elements of the Chevalley basis, we have\\
\begin{equation}\label{t3cetrnaesta}\sum_{\beta \in (\Phi')^+ \atop \beta_1-\beta \in \Phi \setminus \Phi'}(a_{\beta_1-\beta,i}b_{\beta,1}- a_{\beta_1-\beta,1}b_{\beta,i})R_{\beta,\beta_1-\beta}+a_{\beta_1,i}\beta_1(h) = 0,\end{equation}
for all $\beta_1 \in \Phi \setminus \Phi'$. Now, we recall that $h \in \mathfrak{I}$, hence $[h,x_{\beta_1}] \neq 0$ for all $\beta_1 \in \Phi \setminus \Phi'$. Hence, equation (\ref{t3cetrnaesta}) yields:\\
\begin{equation}\label{t3petnaesta}a_{\beta_1,i} = \frac{-\sum_{\beta \in (\Phi')^+\atop \beta_1-\beta \in \Phi \setminus \Phi'}(a_{\beta_1-\beta,i} b_{\beta,1}-a_{\beta_1-\beta,1}b_{\beta,i})R_{\beta,\beta_1-\beta}}{\beta_1(h)},\end{equation}
for all $\beta_1 \in \Phi \setminus \Phi'$, and all $i=2,...,m$. Recall that we have partial ordering on $\Phi$, given by $\alpha \prec \beta$ if and only if $\beta - \alpha$ is a positive root. Note that in the equations (\ref{t3petnaesta}) we expressed coefficient $a_{\beta_1,i}$ via the coefficients $a_{\phi,i}$ for $\phi \prec \beta_1$. From here, it is clear that these are linearly independent conditions, hence we have $(m-1)(|\Phi \setminus \Phi'|)$ conditions. The space of solutions of the system (18) has therefore dimension\\
\begin{equation*} m(|\Phi \setminus \Phi'|+|\Delta \setminus \Delta'|)-(m-1)(|\Phi \setminus \Phi'|)=\end{equation*}
\begin{equation*} = \dim(G)-\dim(G')+(m-1)(|\Delta|-|\Delta'|).\end{equation*}
 As equations (\ref{t3trinaesta}) could only lower the dimension of the space of solutions, we obtain\\
\begin{equation*} \dim(T_{(x_1'+h,x_{2,0}',...,x_{m,0}')}C_m(L))-\dim(T_{(x_1',x_{2,0}',...,x_{m,0}')}C_m(L')) \leq \dim(G)-\dim(G')+(m-1) (|\Delta|-|\Delta'|),\end{equation*}
ending the proof.\end{proof}

By the same argument as in the last part of Corollary 3.3 and using Theorem 3.4 instead of Theorem 3.1 we show:\\

\begin{corollary}\label{fin} If $L'$ is a simple Lie subalgebra of $L$ such that the Dynkin diagram of $L'$ is a subdiagram of the Dynkin diagram of $L$ and $C'$ is a component of $C_m(L')$ containing only nilpotent $m$-tuples, and $\dim(T_{(x_1',...,x_m')}C_m(L'))= \dim(C')$ for some point $(x_1',...,x_m') \in C'$, then $C_m(L)$ is reducible.\\\end{corollary}

\section{Application of Theorem on Adding Diagonals to reducibility of $C_m(L)$}

In this section we prove Theorem 1.1, using Theorem on Adding Diagonals. The proof is split in various sections, depending on $m$ and the type of $L$.\\

\subsection{Reducibility of $C_m(L)$ for $m \geq 4$ and $L$ of types $A,D,E$}

This was done in \cite{Guralnick} for $\mathfrak{gl_l}$, while our approach works also for types $D$ and $E$. We will prove reducibility of commuting varieties by finding a set $S'$ satisfying the assumptions of Theorem 3.4. We first consider $C_4(\mathfrak{sl}_4)$ and compute the $T$-space to Guralnick's example \cite{Guralnick}. Our setting slightly differs from the one found in literature \cite{Guralnick,Jel} not only because we consider $\mathfrak{sl}_4$ instead of $\mathfrak{gl}_4$, but also because our smooth point is different, as we want the centralizer $C(x_1)$ to be as small as possible for $S'$ to be $\mathrm{Stab}_{\mathrm{SL}_4}(x_1)$-stable. Therefore we provide the proof for the sake of completeness.\\ 

We take\\
\[x_1 = \left[ \begin{array}{cccc}
0 & 0 & 1 & 0\\
0 & 0 & 0 & 1\\
0 & 0 & 0 & 0\\
0 & 0 & 0 & 0
\end{array}
\right],
x_2 = \left[ \begin{array}{cccc}
0 & 0 & 0 & 1\\
0 & 0 & 0 & 0\\
0 & 0 & 0 & 0\\
0 & 0 & 0 & 0
\end{array}
\right],
x_3 = \left[\begin{array}{cccc}
0 & 0 & 0 & 0\\
0 & 0 & 1 & 0\\
0 & 0 & 0 & 0\\
0 & 0 & 0 & 0
\end{array}
\right],
x_4 = \left[ \begin{array}{cccc}
0 & 0 & 0 & 0\\
0 & 0 & 0 & 1\\
0 & 0 & 0 & 0\\
0 & 0 & 0 & 0
\end{array}
\right].
\]\\

The $T$-space to $C_4(\mathfrak{sl}_4)$ at the point $(x_1,x_2,x_3,x_4)$ is given by\\
\[ \left(\left[ \begin{array}{cccc}
a_1 & c_6 & a_3 & a_4\\
b_1 & -c_{11} & a_7 & a_8\\ 
0 & 0 & -a_1 & -c_6\\
0 & 0 & -b_1 & c_{11}
\end{array} \right], \left[ \begin{array}{cccc}
b_1 & -c_{11} & b_3 & b_4\\
0 & 0 & b_7 & b_8\\
0 & 0 & 0 & -a_1\\
0 & 0 & 0 & -b_1
\end{array} \right], \left[ \begin{array}{cccc}
0 & 0 & c_3 & c_4\\
a_1 & c_6 & c_7 & c_8\\
0 & 0 & -c_6 & 0\\
0 & 0 & c_{11} & 0
\end{array} \right], \left[
\begin{array}{cccc}
0 & 0 & d_3 & d_4\\
b_1 & -c_{11} & d_7 & d_8\\
0 & 0 & 0 & -c_6\\
0 & 0 & 0 & c_{11}
\end{array}\right]
\right),\]\\
where $\{a_1,c_6,a_3,a_4,b_1,c_{11},a_7,a_8,b_3,b_4,b_7,b_8,c_3,c_4,c_7,c_8,d_3,d_4,d_7,d_8\} \subset F$ are free parameters.\\

Next, we observe that for $m \geq 4$ the $T$-space $T_{(x_1,x_2,x_3,x_4,0,...,0)}C_m(\mathfrak{sl}_4)$ consists of all $m$-tuples $(z_1,...,z_m)$ where $(z_1,...,z_4) \in T_{(x_1,x_2,x_3,x_4)}C_4(\mathfrak{sl}_4)$ and $z_j \in \cap_{i=1}^4C(x_i)$ for $j \geq 5$.\\

An easy computation shows that $z \in \cap_{i=1}^4 C(x_i)$ is of the form\\
$$z = \begin{bmatrix}
0 & 0 & \alpha_{1} & \alpha_{2}\\
0 & 0 & \alpha_{3} & \alpha_{4}\\
0 & 0 & 0 & 0\\
0 & 0 & 0 & 0
\end{bmatrix}$$\\
so the dimension of the $T$-space to $C_m(\mathfrak{sl}_4)$ associated to the point $(x_1,x_2,x_3,x_4,0,...,0)$ is\\
\begin{equation*} \dim(T_{(x_1,x_2,x_3,x_4,0,...,0)}C_m(\mathfrak{sl}_4)) = 20+4(m-4) = 4m+4.\\ \end{equation*} 

Let\\
\[ S' = \left\{ \left( \left[\begin{array}{cccc}
0 & 0 & \alpha_{2, 1} & \alpha_{2, 2}\\
0 & 0 & \alpha_{2, 3} & \alpha_{2, 4}\\
0 & 0 & 0 & 0\\
0 & 0 & 0 & 0
\end{array}
\right], ....,
\left[ \begin{array}{cccc}
0 & 0 & \alpha_{m,1} & \alpha_{m,2}\\
0 & 0 & \alpha_{m,3} & \alpha_{m,4}\\
0 & 0 & 0 & 0\\
0 & 0 & 0 & 0
\end{array}
\right]\right) : \alpha_{i,1},\alpha_{i,2},\alpha_{i,3},\alpha_{i,4} \in F \right\}.
\]
\\
It is easy to check that $S'$ is $Stab_{SL_4}(x_1)$-stable, irreducible and of dimension $4m-4$.\\

Let $C' = SL_4 \cdot (\{x_1\} \times S\} \subset C_m(\mathfrak{sl}_4)$. The dimension of $C_{\mathfrak{sl}_4}(x_1)$ is 7, so it follows by Lemma \ref{Lemma11} that\\
$$\dim(C') =4m + 4.$$\\
\\

We see that $\dim(C')$ is equal to the dimension of the $T$-space to $C_m(\mathfrak{sl}_4)$ associated to the point $(x_1,...,x_n,0,...,0)$ so $\overline{C'}$ is an irreducible component. Moreover, Corollary \ref{fin} implies:\\

\begin{theorem}The variety $C_{m}(L)$ is reducible for $m \geq 4$ and $L$ being a simple Lie algebra of types $A, D, E_{6}, E_7, E_{8}.$\\ \end{theorem}

\subsection{Reducibility of $C_{m}(L)$ for $L$ of types $B,C,F$, and $m \geq 3$}

We will use the representation of $\mathfrak{sp}_{2l}$ as in \cite{Hump}, i.e. its elements are of the form $\begin{bmatrix}
A & B\\
C & -A^{T}
\end{bmatrix}$, where $A$ is any $l \times l$ matrix, while $B$ and $C$ are symmetric $l \times l$ matrices. To prove reducibility of commuting varieties we use the same approach as in the previous section.\\

We take\\
\[ x_1 = \left[ \begin{array}{cccc}
0 & 1 & 0 & 0\\
0 & 0 & 0 & 0\\
0 & 0 & 0 & 0\\
0 & 0 & -1 & 0
\end{array}\right], x_2 = \left[ \begin{array}{cccc}
0 & 0 & 1 & 0\\
0 & 0 & 0 & 0\\
0 & 0 & 0 & 0\\
0 & 0 & 0 & 0
\end{array}
\right],
x_{3} = \left[ \begin{array}{cccc}
0 & 0 & 0 & 0\\
0 & 0 & 0 & 0\\
0 & 0 & 0 & 0\\
0 & 1 & 0 & 0
\end{array}
\right].
\]\\

The $T$-space associated to the point $(x_{1},x_{2},x_{3}) \in C_3(\mathfrak{sp}_4)$ is given by\\
\scalebox{0.75}{$\left\{\left(\left[ \begin{array}{cccc}
a_1 & a_2 & a_5 & a_6\\
0 & -a_1 & a_6 & 0\\
0 & a_9 & -a_1 & 0\\
a_9 & a_{10} & -a_2 & a_1
\end{array} \right],\left[ \begin{array}{cccc}
-a_9 & c_2 & c_5 & -a_1\\
0 & 0 & -a_1 & 0\\
0 & 0 & a_9 & 0\\
0 & c_{10} & -c_2 & 0
\end{array} \right],\left[\begin{array}{cccc}
0 & b_2 & b_5 & 0\\
0 & -a_6 & 0 & 0\\
0 & a_1 & 0 & 0\\
a_1 & b_{10} & -b_2 & a_6
\end{array}\right]\right),a_1,a_2,a_5,a_6,a_9,a_{10},c_2,c_5,c_{10},b_2,b_5,b_{10} \in F\right\},$}\\
and a matrix commuting with $x_1,x_2,x_3$ is of the form\\
$$z = \begin{bmatrix}
0 & \alpha_2 & \alpha_5 & 0\\
0 & 0 & 0 & 0\\
0 & 0 & 0 & 0\\
0 & \alpha_{10} & -\alpha_2 & 0
\end{bmatrix},$$\\
so\\
\begin{equation*}\dim(T_{(x_1.x_2,x_3,0,...,0)}C_m(\mathfrak{sp}_4)) = 12+3(m-3) = 3m+3.\\ \end{equation*}

Let\\
\[S' = \left\{ \left(
\left[ \begin{array}{cccc}
0 & a_{2,2} & a_{2,3} & 0\\
0 & 0 & 0 & 0\\
0 & 0 & 0 & 0\\
0 & a_{2,4} & - a_{2,2} & 0
\end{array} \right]\\
, ...., \left[ \begin{array}{cccc}
0 & a_{m,2} & a_{m,3} & 0\\
0 & 0 & 0 & 0\\
0 & 0 & 0 & 0\\
0 & a_{m,4} & - a_{m,2} & 0
\end{array} \right]\\
\right): a_{i,2},a_{i,3},a_{i,4} \in F\right\}.
\]\\
Clearly, $S' \subset C_{m-1}(C_{\mathfrak{sp}_4}(x_1))$ is an affine space of commuting matrices of dimension $3(m-1)$, hence irreducible.\\

The stabilizer $Stab_{\mathrm{SP}_4}(x_1)$ is\\
\scalebox{0.8}{$\left\{\left[\begin{array}{cccc}
\frac{1}{y_{11}} & \frac{y_{15}}{y_{11}^2} & y_3 & 0\\
0 & \frac{1}{y_{11}} & 0 & 0\\
0 & 0 & y_{11} & 0\\
0 & y_{14} & y_{15} & y_{11}
\end{array}\right]: y_3,y_{15},y_{14} \in F, y_{11}\in F^*\right\} \cup \left\{\left[\begin{array}{cccc}
0 & y_2 & \frac{y_{14}}{y_{10}^2} & \frac{1}{y_{10}}\\
0 & 0 & -\frac{1}{y_{10}} & 0\\
0 & y_{10} & 0 & 0\\
-y_{10} & y_{14} & y_{15} & 0
\end{array}\right]: y_2,y_{14},y_{15} \in F, y_{10} \in F^*\right\}$},\\
and it is easy to verify that $S'$ is $Stab_{SP_4}(x_1)$-stable. The dimension of $C'=  SP_4 \cdot (\{x_1\} \times S') \subset C_4(\mathfrak{sp}_4)$ is\\
$$\dim(C') = \dim(S') + \dim(\mathfrak{sp}_4) - \dim(C(x_1))= 3m+3.$$\\

Since $\dim(C') = \dim(T_{(x_1,x_2,x_3,0,...,0)}C_m(\mathfrak{sp}_4))$, $\overline{C'}$ is a component of $C_m(\mathfrak{sp}_4)$, and $(x_1,x_2,x_3,0,...,0)$ is a smooth point on it. Applying Corollary \ref{fin} we get:\\

\begin{theorem} The variety $C_m(L)$ is irreducible for every $m \geq 3$ and  $L$ of types $F_{4}$ and $B_l$, $C_l$ for $l \geq 2$.\\ \end{theorem}

\subsection{Reducibility of $C_{m}(L)$ for $L$ of type $G_{2}$}

In this section we show reducibility of $C_m(L)$ if $L$ is of type $G_2$, by adapting \cite[Example 7.2]{Ngo3} to our setting. \cite[Example 7.2]{Ngo3} shows that the variety $C_m(N(L))$ of $m$-tuples of commuting nilpotent elements of $L$ is not equidimensional for $m \geq 3$. We will use the same example to show that $C_m(L)$ is reducible for $m \geq 3$, but contrary to \cite{Ngo3}, we will need $T$-spaces to prove this. In the calculation we will use Chevalley basis. Recall that for $\phi,\psi \in \Phi$, such that $\phi+\psi \in \Phi$, we denoted by $R_{\phi,\psi}$ the constant satisfying $[x_{\phi},x_{\psi}] = R_{\phi,\psi}x_{\phi+\psi}$. There is some choice of signs of the constants $R_{\phi,\psi}$. We will use the signs that are used in GAP. Reader may find them in Appendix 2.\\

Let $\Delta = \{\alpha,\beta\}$ be the base of the root system $G_2$, with $\alpha$ being the shorter root and $\beta$ being the longer root. We take $(x_{\beta}+x_{3\alpha+\beta}, x_{2\alpha+\beta},x_{3\alpha+2\beta}) \in C_3(L)$, that was considered in \cite[Example 7.2]{Ngo3}. We denote $x_1 = x_{\beta}+x_{3\alpha+\beta},x_2 = x_{2\alpha+\beta},x_3= x_{3\alpha+2\beta}$.\\

The $T$-space to $C_3(L)$ associated to the point $(x_{\beta}+x_{3\alpha+\beta},x_{2\alpha+\beta},x_{3\alpha+2\beta})$ is\\
\begin{center}$T = \{-\frac{b_{14}}{3}x_{\alpha}+ a_2x_{\beta}+ a_3x_{\alpha+\beta}+ a_4x_{2\alpha+\beta}+ a_5x_{3\alpha+\beta}+ a_6x_{3\alpha+2\beta}+ a_7x_{-\alpha} +b_9x_{-\beta}+b_9x_{-3\alpha-\beta}+ a_{13}h_{\alpha}+ a_{14} h_{\beta},$\end{center}
\begin{center}$2 a_7x_{\alpha}+ b_2x_{\beta}+ b_3x_{\alpha+\beta}+ b_4x_{2\alpha+\beta}+(3 a_3 + b_2)x_{3\alpha+\beta}+ b_6x_{3\alpha+2\beta} -a_{13}x_{-\alpha}+ b_9x_{-\alpha-\beta}+ \frac{2 b_{14}}{3}h_{\alpha}+ b_{14}h_{\beta}$,\end{center}
\begin{center}$c_2x_{\beta}+ \frac{b_{14}}{3}x_{\alpha+\beta}+c_4x_{2\alpha+\beta}+ (a_{14} + c_2)x_{3\alpha+\beta}+ c_6x_{3\alpha+2\beta}+ b_9h_{\alpha}+2 b_9h_{\beta},$\end{center}
\begin{center}$b_{14},a_2,a_3,a_4,a_5,a_6,a_7,b_9,a_{13},a_{14},b_2,b_3,b_4,b_6,c_2,c_4,c_6 \in F\}.$\end{center}

An easy computation shows that\\
\begin{equation*} \cap_{j=1}^3 C(x_j) = \{i_3(x_{\beta}+x_{3\alpha+\beta})+ i_7x_{2\alpha+\beta}+i_{11}x_{3\alpha+2\beta}, i_3,i_7,i_{11} \in F\}.\\ \end{equation*}

The same argument as in Section 4.1 shows that\\
\begin{equation*} \dim(T_{(x_1,x_2,x_3,0,...,0)}C_m(L)) = 17+3(m-3) = 3m+8.\end{equation*}

Let $\mathcal{L}_2$ be the Lie subalgebra of $L$, generated by $x_{\beta}+x_{3\alpha+\beta},x_{2\alpha+\beta},x_{3\alpha+2\beta}$ and let $C' = G \cdot \mathcal{L}_2^m$. Then $\dim(C') = 3m+8$ by \cite[Example 7.2]{Ngo3}. Since the authors of \cite{Ngo3} have not published that paper yet, we give an outline of their argument: $C'$ is the image of the map\\
\begin{equation*} G \times \mathcal{L}_2^m \to C'\end{equation*}
\begin{equation*} (g,x_1,...,x_m) \mapsto (g \cdot x_1,...,g \cdot x_m)\end{equation*}
whose fibre at the general point is the normalizer of $\mathcal{L}_2$ in $L$ which is $6$-dimensional. Theorem on Dimension of Fibres then implies that $\dim(C') = 14+3m-6=3m+8$.\\
 
Since $C'$ contains only $m$-tuples of nilpotent elements, its closure can not be $\mathrm{Reg}_{m,L}$. We have proven:\\

\begin{theorem} Let $L$ be a simple Lie algebra of type $G_2$. Then $C_{m}(L)$ is reducible for $m \geq 3$.\\ \end{theorem}

\begin{corollary} For any $m \geq 4$ and any semisimple Lie algebra $L$ that is not of type $\mathfrak{sl}_2 \times... \times \mathfrak{sl}_2 \times \mathfrak{sl}_3 \times ... \times \mathfrak{sl}_3$ the variety $C_m(L)$ is reducible.\\ \end{corollary}

\subsection{Reducibility of $C_{3}(L)$ for $L$ of type $D_l$ with $l \geq 10$}

We use the representation of $\mathfrak{so}_{2l}$ as in \cite{Hump}: Elements of $\mathfrak{so}_{2l}$ are of the form 
$\begin{bmatrix}
A & B\\
C & - A^{T}
\end{bmatrix}$
where $A$ is an arbitrary $l \times l$ matrix, and $B$ and $C$ are antisymmetric $l \times l$ matrices. Note that the dimension of $\mathfrak{so}_{2l}$ is $2l^{2} -l$.\\

Take\\
\[ x_{1} = \left[\begin{array}{cccc}
0 & I & 0 & I\\
0 & 0 & - I & 0\\
0 & 0 & 0 & 0\\
0 & 0 & - I & 0\\
\end{array} \right],\] 
where each block is of size $s \times s$. The centralizer of $x_{1}$ consists of all matrices of the form
$\begin{bmatrix}
A_{1} & A_{2} & A_5 & A_6\\
0 & A_{4} & - A_6^{T} & A_1+A_4^{T}\\
0 & 0 & A_1 & 0\\
0 & A_1 -A_4 & -A_2^T & A_4\\
\end{bmatrix}$
where $A_1,A_4,A_5$ are all antisymmetric, $A_2$ is arbitrary and $A_2+A_6$ is symmetric. The dimension of the centralizer is $3s^{2}- s$.\\
\\
Let $S'$ be the set of pairs of commuting matrices of the type
$\begin{bmatrix}
0 & A_{2} & A_5 & A_6\\
0 & 0 & - A_6^T & 0\\
0 & 0 & 0 & 0\\
0 & 0 & -A_2^T & 0
\end{bmatrix}$\\
where $A_5$ is antisymmetric and $A_2 +A_6$ is symmetric. The space of pairs of matrices of this shape has dimension  $4s^{2}$.Now we consider the commutativity condition. Commutator of such matrices is\\
\[
[A,B] = \left[
\left[
\begin{array}{cccc}
0 & A_{2} & A_5 & A_6\\
0 & 0 & -A_6^T & 0\\
0 & 0 & 0 & 0\\
0 & 0 & -A_2^T & 0
\end{array}
\right]
,
\left[
\begin{array}{cccc}
0 & B_2 & B_5 & B_6\\
0 & 0 & -B_6^T & 0\\
0 & 0 & 0 & 0\\
0 & 0 & -B_2^T & 0
\end{array}
\right]
\right]
=\]\\
\[= \\
\left[ \begin{array}{cccc}
0 & 0 & -A_2B_6^T - A_6B_2^T+B_2A_6^T+B_6A_2^T & 0\\
0 & 0 & 0 & 0\\
0 & 0 & 0 & 0\\
0 & 0 & 0 & 0
\end{array}
\right].
\\
\]
\\
\\
This commutator lies in $\mathfrak{so}_{2l}$, which means that the matrix $-A_2B_6^T - A_6B_2^T+B_2A_6^T+B_6A_2^T$ is antisymmetric, so the condition $[A,B] = 0$ is described by at most $\frac{s(s- 1)}{2}$ equations, so the dimension of $S'$ is at least $4s^{2} - \frac{s(s- 1)}{2} = \frac{7}{2}s^{2} + \frac{s}{2}$. One can easily check that $S'$ is $\mathrm{Stab}_{\mathrm{SO}_{4s}}(x_1)$-stable. Let $C' = \mathrm{SO}_{4s} \cdot (\{x_1\} \times S')$. As in the proof of Lemma \ref{Lemma11} we have\\
\begin{equation*} \dim(C') \geq \dim(S') + \dim(SO_{4s}) - \dim(C(x_{1})) \geq\\ \end{equation*}
\begin{equation*} \geq \frac{7}{2}s^{2} + \frac{s}{2} + \dim(SO_{4s}) - 3s^{2} + s = \frac{s^{2} + 3s}{2} + \dim(SO_{4s}).\\ \end{equation*}

Now recall that the dimension of the regular component of $C_3(\mathfrak{so}_{4s})$ is $\dim(SO_{4s}) + 2 \cdot 2s$. If  $s$ is such that \begin{equation*} \frac{s(s + 3)}{2} + \dim(SO_{4s}) \geq 4s + \dim(SO_{4s}), \end{equation*}\\
then $C_{3}(\mathfrak{so}_{4s})$ is reducible. The above inequality is equivalent to $s(s -5) \geq 0$, i.e. to $l \geq 10$. It follows that:\\

\begin{theorem} For every $l \geq 10$ the variety $C_{3}(\mathfrak{so}_{2l})$ is reducible.\\ \end{theorem}

\subsection{Reducibility of $C_{m}(L)$ for $L$ of type $E_{7}$ and $E_{8}$}

To prove reducibility of $C_m(L)$, where $L$ is of type $E_7$ or $E_8$, we will use results about $\mathbb{Z}$-gradings on simple Lie algebras induced by $\mathfrak{sl}_2$-triples. We first recall these results. For more information on $\mathbb{Z}$-gradings, we refer to \cite[Chapter 5]{Jan}. Let $h$ be a semisimple element of an $\mathfrak{sl}_2$-triple and let $L_i = \{x \in L: [h,x] = ix\}$. We define $d$ as the highest $i$ such that $L_{i} \neq \{0\}$. We have (see \cite[4.10]{Jan} and \cite[Chapter 3.4]{Col}) that $[L_i,L_j] \subset L_{i+j}$ and that $L = \oplus_{i=-d}^dL_i$. Moreover, by \cite[Proposition 5.8.]{Jan} we have that $C(x) = \oplus_{i = 0}^{d}(L_{i} \bigcap C(x))$.\\

If we have $L$ of type $E_7$ or $E_8$, then the nilpotent orbits for the action of $G$ on $L$ are known and classified. For the classification of nilpotent orbits, we refer to \cite{Col}. In the case of $E_7$ the orbits are parametrized by certain 7-tuples $(a_1,...,a_7)$ where $a_{i} \in \{0,1,2\}$, for all $1 \leq i \leq 7$. The list of all orbits can be found in, for example, \cite[Table 22.1.2]{Liebeck}. We take the orbit denoted by (0,2,0,0,0,0,0). Using GAP (see Appendix), we fix a representative of this orbit $x = x_{\alpha_2+\alpha_3+2\alpha_4+\alpha_5+\alpha_6} + x_{\alpha_1+\alpha_2+2\alpha_3+2\alpha_4+\alpha_5} + x_{\alpha_1+\alpha_2+\alpha_3+\alpha_4+\alpha_5+\alpha_6+\alpha_7}+ x_{\alpha_1+\alpha_2+\alpha_3+2\alpha_4+2\alpha_5+\alpha_6}+ x_{\alpha_2+\alpha_3+2\alpha_4+2\alpha_5+\alpha_6+\alpha_7}$, where $\alpha_i$ are base roots of $\Phi$, labelled as in \cite[Theorem 11.4]{Hump}. Using GAP or by \cite[Table 22.1.2]{Liebeck} we get that our $x$ has centralizer of dimension 49. With GAP we compute that $h = 4h_{\alpha_1}+7h_{\alpha_2}+8h_{\alpha_3}+12h_{\alpha_4}+9h_{\alpha_5}+6h_{\alpha_6}+3h_{\alpha_7}$ is the semisimple element of an $\mathfrak{sl}_2$-triple that contains $x$. (Note that such an $\mathfrak{sl}_2$-triple always exists due to Jacobson-Morozov Theorem, see for example \cite[page 36]{Col}.)\\

In our case we get, using GAP, that $d = 4$, $L_{1} \cap C(x)$ and $L_3 \cap C(x)$ are trivial, $\dim(C(x) \bigcap L_2) = 28$ and $\dim(C(x) \bigcap L_4) = 7$. An interested reader may find explicit bases of $L_2 \cap C(x)$ and $L_4 \cap C(x)$ in Appendix 3. This means the only equations for commutativity we could possibly get in $(L_2 \oplus L_4) \cap C(x)$ are the ones arising from $[L_{2},L_{2}]$. This means at most 7 equations. We define\\
\begin{equation*} S = \{(y,z): y,z \in L_{2} \oplus L_{4}, [x,z] = [x,y] = [y,z] = 0\}.\end{equation*}
The dimension of $S$ is at least $2 \cdot 35 -7= 63$, so the dimension of $C= G \cdot (\{x\} \times S)$ is at least $133 +  63 - 49 = 147 = \dim(\mathrm{Reg}_{3,L})$. Since our $C$ contains only triples of nilpotent elements, its closure can not be the regular component. \\

This proves:\\

\begin{theorem} The variety $C_3(L)$ is reducible if $L$ is a simple Lie algebra of type $E_7$.\\ \end{theorem}

Moreover, applying Corollary 3.3, we get:\\

\begin{corollary} The variety $C_3(L)$ is reducible if $L$ is a simple Lie algebra of type $E_8$.\\ \end{corollary}

\textbf{Appendix 1: GAP code for type $E$}\\

\noindent L:= SimpleLieAlgebra("E",7,Rationals);\\
N := NilpotentOrbit(L,[0,2,0,0,0,0,0]);\\
x := SL2Triple(N1)[3];\\
S:= Subalgebra(L,[x]);\\
C:= Centralizer(L,S);\\
h := SL2Triple(N1)[2];\\
G := SL2Grading(L,h);\\
L2 := G[1][2];\\
VL2 := VectorSpace(Rationals,L2);\\
L2new := Intersection(C,VL2);\\
B2 := Basis(L2new);\\
L4 := G[1][4];\\
VL4 := VectorSpace(Rationals,L4);\\
L4new := Intersection(VL4,C);\\
B4 := Basis(L4new);\\

\textbf{Appendix 2: Signs of the constants $R_{\phi,\psi}$ for type $G_2$}\\

Recall that our notation is $[x_{\phi},x_{\psi}] = R_{\phi,\psi} x_{\phi+\psi}$ if $\phi,\psi,\phi+\psi \in \Phi$. The choice of signs of $R_{\phi,\psi}$ is not unique. We here provide a choice we use. If $\psi = -\phi$, then $[x_{\phi},x_{\psi}] = h_{\phi}$ and in those cases we say $R_{\phi,\psi}$ is undefined and in those places we write $\#$. Also, if $\phi+\psi \notin \Phi$, then $[x_{\phi},x_{\psi}]=0$, and in those cases we denote $R_{\phi,\psi}=0$. In the following table $R_{\phi,\psi}$ is the constant written in the row corresponding to $\phi$ and the column corresponding to $\psi$.\\

\scalebox{0.8}{\begin{tabular}{|c|c|c|c|c|c|c|c|c|c|c|c|c|}
\hline
$R_{\phi,\psi}$ & $\alpha$ & $-\alpha$ & $\beta$ & $-\beta$ & $\alpha+\beta$ & $-\alpha-\beta$ & $2\alpha+\beta$ & $-2\alpha-\beta$ & $3\alpha+\beta$ & $-3\alpha-\beta$ & $3\alpha+2\beta$ & $-3\alpha-2\beta$\\
\hline
$\alpha$ & 0 & \# & -1 & 0 & -2 & 3 & -3 & 2 & 0 & 1 & 0 & 0\\
\hline
$-\alpha$ & $\#$ & 0 & 0 & 1 & -3 & 2 & -2 & 3 & -1 & 0 & 0 & 0\\
\hline
$\beta$ & 1 & 0 & 0 & $\#$ & 0 & -1 & 0 & 0 & -1 & 0 & 0 & 1\\
\hline
$-\beta$ & 0 & -1 & $\#$ & 0 & 1 & 0 & 0 & 0 & 0 & 1 & -1 & 0\\
\hline
$\alpha+\beta$ & 2 & 3 & 0 & -1 & 0 & $\#$ & 3 & -2 & 0 & 0 & 0 & -1\\
\hline
$-\alpha-\beta$ & -3 & -2 & 1 & 0 & $\#$ & 0 & 2 & -3 & 0 & 0 & 1 & 0\\
\hline
$2\alpha+\beta$ & 3 & 2 & 0 & 0 & -3 & -2 & 0 & $\#$ & 0 & -1 & 0 & 1\\
\hline
$-2\alpha-\beta$ & -2 & -3 & 0 & 0 & 2 & 3 & $\#$ & 0 & 1 & 0 & -1 & 0\\
\hline
$3\alpha+\beta$ & 0 & 1 & 1 & 0 & 0 & 0 & 0 & -1 & 0 & $\#$ & 0 & -1\\
\hline
$-3\alpha-\beta$ & -1 & 0 & 0 & -1 & 0 & 0 & 1 & 0 & $\#$ & 0 & 1 & 0\\
\hline
$3\alpha+2\beta$ & 0 & 0 & 0 & 1 & 0 & -1 & 0 & 1 & 0 & -1 & 0 & $\#$\\
\hline
$-3\alpha-2\beta$ & 0 & 0 & -1 & 0 & 1 & 0 & -1 & 0 & 1 & 0 & $\#$ & 0\\
\hline
\end{tabular}}\\
\\
\\

\newpage

\textbf{Appendix 3: Elements of bases of $L_{2} \bigcap C(x)$ and $L_4 \cap C(x)$ from Section 4.5}\\

Elements of a basis of $L_2 \cap C(x)$ are:\\

\begin{enumerate} 
\item $x_{\alpha_2}$
\item $x_{\alpha_2+\alpha_4}$
\item $x_{\alpha_2+\alpha_3+\alpha_4}+x_{\alpha_2+\alpha_4+\alpha_5+\alpha_6}$
\item $x_{\alpha_2+\alpha_4+\alpha_5}$
\item  $x_{\alpha_1+\alpha_2+\alpha_3+\alpha_4}+x_{\alpha_2+\alpha_3+2\alpha_4+\alpha_5}$
\item $x_{\alpha_2+\alpha_3+\alpha_4+\alpha_5}$
\item $x_{\alpha_2+\alpha_4+\alpha_5+\alpha_6}+x_{\alpha_1+\alpha_2+\alpha_3+\alpha_4+\alpha_5}$
\item $x_{\alpha_2+\alpha_3+2\alpha_4+\alpha_5}+x_{\alpha_2+\alpha_4+\alpha_5+\alpha_6+\alpha_7}$
\item $x_{\alpha_2+\alpha_3+\alpha_4+\alpha_5+\alpha_6}$
\item $x_{\alpha_1+\alpha_2+\alpha_3+2\alpha_4+\alpha_5}$
\item $x_{\alpha_1+\alpha_2+\alpha_3+\alpha_4+\alpha_5+\alpha_6} - x_{\alpha_2+\alpha_3+\alpha_4+\alpha_5+\alpha_6+\alpha_7}$
\item $x_{\alpha_2+\alpha_3+2\alpha_4+\alpha_5+\alpha_6}+x_{\alpha_1+\alpha_2+\alpha_3+\alpha_4+\alpha_5+\alpha_6+\alpha_7}$
\item $x_{\alpha_2+\alpha_3+\alpha_4+\alpha_5+\alpha_6+\alpha_7}-x_{\alpha_2+\alpha_3+2\alpha_4+2\alpha_5+\alpha_6}$
\item $x_{\alpha_1+\alpha_2+2\alpha_3+2\alpha_4+\alpha_5}$
\item $x_{\alpha_1+\alpha_2+\alpha_3+2\alpha_4+\alpha_5+\alpha_6}-x_{\alpha_2+\alpha_3+2\alpha_4+\alpha_5+\alpha_6+\alpha_7}$
\item $x_{\alpha_2+\alpha_3+2\alpha_4+\alpha_5+\alpha_6+\alpha_7}-x_{\alpha_1+\alpha_2+\alpha_3+2\alpha_4+2\alpha_5+\alpha_6+\alpha_7}$
\item $x_{\alpha_1+\alpha_2+2\alpha_3+2\alpha_4+\alpha_5+\alpha_6}-x_{\alpha_2+\alpha_3+2\alpha_4+2\alpha_5+2\alpha_6+\alpha_7}$
\item $x_{\alpha_1+\alpha_2+\alpha_3+2\alpha_4+2\alpha_5+\alpha_6}$
\item $x_{\alpha_1+\alpha_2+\alpha_3+2\alpha_4+\alpha_5+\alpha_6+\alpha_7}$
\item $x_{\alpha_2+\alpha_3+2\alpha_4+2\alpha_5+\alpha_6+\alpha_7}$
\item $x_{\alpha_1+\alpha_2+2\alpha_3+2\alpha_4+2\alpha_5+\alpha_6}$
\item $x_{\alpha_1+\alpha_2+2\alpha_3+2\alpha_4+\alpha_5+\alpha_6+\alpha_7}+x_{\alpha_1+\alpha_2+2\alpha_3+3\alpha_4+2\alpha_5+\alpha_6}$
\item $x_{\alpha_2+\alpha_3+2\alpha_4+2\alpha_5+2\alpha_6+\alpha_7}-x_{\alpha_1+\alpha_2+2\alpha_3+2\alpha_4+2\alpha_5+\alpha_6+\alpha_7}$
\item $x_{\alpha_1+\alpha_2+2\alpha_3+3\alpha_4+2\alpha_5+\alpha_6}-x_{\alpha_1+\alpha_2+\alpha_3+2\alpha_4+2\alpha_5+2\alpha_6+\alpha_7}$
\item $x_{\alpha_1+\alpha_2+2\alpha_3+3\alpha_4+2\alpha_5+\alpha_6+\alpha_7}$
\item $x_{\alpha_1+\alpha_2+2\alpha_3+2\alpha_4+2\alpha_5+2\alpha_6+\alpha_7}$
\item $x_{\alpha_1+\alpha_2+2\alpha_3+3\alpha_4+2\alpha_5+2\alpha_6+\alpha_7}$
\item $x_{\alpha_1+\alpha_2+2\alpha_3+3\alpha_4+3\alpha_5+2\alpha_6+\alpha_7}$
\end{enumerate}

Elements of a basis of $L_{4} \bigcap C(x)$ are:

\begin{enumerate}
\item $x_{\alpha_1+2\alpha_2+2\alpha_3+3\alpha_4+2\alpha_5+\alpha_6}$
\item $x_{\alpha_1+2\alpha_2+2\alpha_3+3\alpha_4+2\alpha_5+\alpha_6+\alpha_7}$
\item $x_{\alpha_1+2\alpha_2+2\alpha_3+3\alpha_4+2\alpha_5+2\alpha_6+\alpha_7}$
\item $x_{\alpha_1+2\alpha_2+2\alpha_3+3\alpha_4+3\alpha_5+2\alpha_6+\alpha_7}$
\item $x_{\alpha_1+2\alpha_2+2\alpha_3+4\alpha_4+3\alpha_5+2\alpha_6+\alpha_7}$
\item $x_{\alpha_1+2\alpha_2+3\alpha_3+4\alpha_4+3\alpha_5+2\alpha_6+\alpha_7}$
\item $x_{2\alpha_1+2\alpha_2+3\alpha_3+4\alpha_4+3\alpha_5+2\alpha_6+\alpha_7}$
\end{enumerate}

\end{document}